%% file: SISC_SMITD.tex
\documentclass[hidelinks,onefignum,onetabnum]{siamart250211}

\input{ex_shared}

\ifpdf
\hypersetup{
	pdftitle={An Example Article},
	pdfauthor={D. Doe, P. T. Frank, and J. E. Smith}
}
\fi

\usepackage{multirow}
\usepackage{graphicx}
\usepackage{subfig}

\newtheorem{remark}{Remark}[section]
\newtheorem{defn}{Definition}[section]

\begin{document}
	
	\maketitle
	
	\begin{abstract}
		Recently, triple decomposition has attracted increasing attention for decomposing third-order tensors into three factor tensors. However, this approach is limited to third-order tensors and enforces uniformity in the lower dimensions across all factor tensors, which restricts its flexibility and applicability.  To address these issues, we propose the Multiple  decomposition, a novel framework that generalizes triple decomposition to  arbitrary order tensors and allows the short dimensions of the factor tensors to differ. We  establish its connections with other classical tensor decompositions. Furthermore,  implicit neural representation (INR) is employed to continuously represent the factor tensors in Multiple decomposition, enabling the method to generalize to non-grid data. We refer to this INR-based Multiple decomposition as  Implicit Multiple Tensor Decomposition (IMTD). Then, the Proximal Alternating Least Squares (PALS)  algorithm is utilized to solve the IMTD-based tensor reconstruction models. Since the objective function in IMTD-based models often lacks the Kurdyka–Łojasiewicz (KL) property, we establish a KL-free convergence analysis for the algorithm.  Finally, extensive numerical experiments further validate the effectiveness of the proposed method.
	\end{abstract}
	
	
	
	\section{Introduction}
	With the rapid development of big data processing in recent years, tensors have received significant attention and have become one of the most powerful tools for representing multi-dimensional data. Currently, tensors play a pivotal role in various fields, including remote sensing image processing \cite{ Chen2022, B2018}, computer vision \cite{Liu2013, Zhang2021}, deep learning \cite{Sa2024},  and signal processing \cite{Lin2020}. 
	
	Numerous real-world tensor datasets inherently exhibit low-rank structures, reflecting an underlying low-dimensional nature \cite{Luo2024}. This property has been widely observed in regular grid-based data such as images \cite{Zhao1}, videos \cite{Be2017}, and wireless channel measurements \cite{Ar2021}. Recently, low-rank characteristics have also been identified in tensor representations of irregular, non-grid data, such as point cloud  \cite{Zhu2022, Luo2024}.  Presently, various tensor decompositions and corresponding tensor ranks have been proposed to explore the underlying low-rank structure of  data. Among the most classical  are the CANDECOMP/PARAFAC (CP) decomposition and Tucker decomposition. The CP decomposition  represents a tensor as a sum of rank-one tensor components, and the minimum number of these rank-one tensors is called the CP rank \cite{Kolda}. Under mild conditions, the CP decomposition is  unique, a property that has enabled its widespread use in diverse fields, such as chemical component analysis \cite{Wang2024} and spectral unmixing \cite{Im2020}. However, computing the exact CP rank is NP-hard \cite{Si2008}, making its accurate determination computationally intractable in general. Tucker decomposition factorizes a tensor into a core tensor and a collection of factor matrices associated with each mode. The Tucker rank  is the tuple of the matrix ranks of the tensor's mode-$n$ unfoldings \cite{Kolda}. In fact, CP decomposition can be regarded as a special case of Tucker decomposition in which the core tensor is super-diagonal. Currently, Tucker decomposition is widely used in applications such as data compression \cite{Zheng2025} and image restoration \cite{CSTF}. However, the core tensor  usually contains a large number of parameters, resulting in high storage and computational complexity \cite{TriB}.
	
	Another family  is the tensor network decompositions, including Tensor Train (TT)  \cite{TT}, Tensor Ring (TR)  \cite{TR} and Fully Connected Tensor Network (FCTN) \cite{FCTN} decompositions. These approaches have been successfully applied to image reconstruction \cite{Jin2022}  and data compression \cite{Ak2024}. However, both TT and TR decompositions only model interactions between adjacent factor tensors,  failing to capture long-range dependencies. Although  FCTN decomposition allows full connectivity among all factors, it comes at the cost of substantially higher storage and computational complexity.
	In recent years, tensor singular value decomposition (t-SVD) \cite{Kilmer2011} has   attracted increasing attention. Within this framework,   the tensor nuclear norm (TNN), as a convex relaxation of tensor tubal rank,  has  been widely applied in fields such as tensor completion \cite{Zhao1} and robust principal component analysis \cite{TSVD}. To enhance flexibility, the transform-based tensor nuclear norm (TTNN) has also been proposed for low-rank approximation \cite{Qiu2021a, Li2022a}.  Wang \textit{et al.} \cite{Wang2025} extended the t-SVD framework to functional settings for  missing slice tensor recovery tasks. Despite its popularity, however,  t-SVD is limited to the decomposition of third-order tensors \cite{MTTD}. 
	
	More recently, Qi \textit{et al.} \cite{TriB} developed the  triple decomposition, which decomposes a third-order tensor into three factor tensors. Each factor tensor has one dominant dimension, which is considered to  inherit the features along the corresponding modes. Currently, triple decomposition has found widespread applications in areas such as  traffic network recovery \cite{Ming2024} and image fusion \cite{Yang2025}. Nonetheless, in contrast to classical tensor decompositions such as CP and Tucker, triple decomposition is specifically designed for third-order tensors, which limits its applicability to higher-order data. Moreover, triple decomposition requires the short dimensions of the factor tensors to be equal, which limits its flexibility and interpretability.

	To address these issues,  we  propose the Multiple decomposition, a method which decomposes an $N$ th-order tensor ($N\ge 3$) into $N$ factor tensors.  Each factor tensor also has only one dominant dimension and is designed to inherit features from the corresponding mode of the original tensor. Unlike triple decomposition, multiple decomposition allows the sizes of the non-dominant dimensions in the factor tensors to be flexible and independently adjustable. Theoretically, we prove that this variability in short dimensions enables a higher compression ratio, improving storage and computational  efficiency. We also establish the connections between Multiple decomposition and other classical tensor decompositions.
	In addition, it is worth noting that most existing tensor decomposition methods are  designed for grid-structured  tensor data, making them ill-suited for non-grid data formats. To further extend the applicability of Multiple decomposition, we incorporate implicit neural representation (INR) \cite{INR} to represent the factor tensors as continuous functions. Specifically, each factor tensor is modeled as an implicit function,  parameterized by a neural network that maps a coordinate  to the associated value in the tensor.  We refer to this INR-based Multiple decomposition as Implicit Multiple Tensor Decomposition (IMTD), which enables the modeling of low-rank structures in irregular and non-grid data. 
	Then, the Proximal Alternating Least Squares (PALS) algorithm is employed to solve the tensor reconstruction problem within the IMTD framework. Under mild conditions, we show that the implicit functions used to represent the  target tensors in IMTD are Lipschitz continuous. Then, we provide a convergence analysis for the PALS algorithm that does not rely on the Kurdyka–Łojasiewicz (KL) property  \cite{Attouch,KL3}. This is  important since the KL property can be difficult to verify for deep neural networks. Finally, extensive numerical experiments are conducted to validate the effectiveness and broad applicability of the proposed method.

	The main contributions of this paper are summarized as follows:
	\begin{itemize}
		\item  We propose the novel Multiple  decomposition, which  decomposes a $N$ th-order tensor ($N\ge 3$) into $N$ factor tensors. In fact, multiple decomposition can be viewed as a generalization of triple decomposition, extending its applicability from third-order tensors to  arbitrary order tensors. Moreover, Multiple  decomposition allows the short dimensions of factor tensors to differ, enhancing flexibility and adaptability of the model.

		\item By integrating Multiple decomposition with implicit neural representation (INR), we develop the Implicit Multiple Tensor Decomposition (IMTD), a framework capable of representing the factor tensors as continuous implicit functions  and performing decomposition on non-grid structured data. 
		
		\item Theoretically,  we establish the connections between the Multiple decomposition and other classical tensor decompositions. Under mild conditions, we prove that the IMTD-induced tensor function is  Lipschitz continuous. Moreover, we  provide a KL-free convergence analysis for the PALS algorithm. Finally, extensive numerical experiments are conducted to demonstrate the effectiveness of the proposed method.
	\end{itemize}
	
	The remaining sections of this paper are organized as follows.   Section \ref{B} outlines the essential notation and definitions.  The proposed IMTD is  introduced in Section \ref{D}. In Section \ref{E}, we present corresponding model, algorithm and convergence analysis. Finally, we provide extensive experiments to verify the effectiveness of IMTD in Section \ref{NN}  and make a conclusion in Section \ref{G}.
	
	\section{Preliminary}\label{B}
	In this paper, we adopt the following notation: scalars are denoted by lowercase letters, such as `$a$'; vectors are denoted by bold lowercase letters such as `$\textbf{a}$'; matrices are represented by bold uppercase letters such as `$\textbf{A}$'; tensors are denoted by calligraphic letters such as `$\mathcal{A}$'; sets and fields  are denoted by hollow letters such as `$\Re$'. The `max', `submax', and `mid' denote the maximum value, the second largest element, and the median operation, respectively.
	
	For  $N$-order tensor $\mathcal{X}\in \Re^{I_{1}\times I_{2} \times...\times I_{N}}$, we denote $\mathcal{X}_{(n)}\in\Re^{I_n\times (I_1...I_{n-1}I_{n+1}...I_N)}$ as the mode-$n$ unfolding matrix of $\mathcal{X}$. 
	The Frobenius norm (F norm) of $\mathcal{X}$ is defined as $\Vert \mathcal{X} \Vert_F=\sqrt{\sum_{i_1=1}^{I_1}\sum_{i_2=1}^{I_2}...\sum_{i_N=1}^{I_N}|\mathcal{X}_{i_1i_2...i_N}|^2}$, and its $\ell_1$ norm is defined as $\Vert \mathcal{X}\Vert_{\ell_1}=\sum_{i_1=1}^{I_1}\sum_{i_2=1}^{I_2}...\sum_{i_N=1}^{I_N}|\mathcal{X}_{i_1i_2...i_N}|$.
	In this paper, F norm is used if there is no special explanation.
	The mode-$n$ product of a tensor $\mathcal{A}\in \Re^{I_{1}\times I_{2},...,\times I_{N}}$ with a matrix $\textbf{U}\in \Re^{J\times I_n}$ is denoted by $\mathcal{A}\times_{n}\textbf{U}$ and is of size $I_{1}\times,...,\times I_{n-1}\times J\times I_{n+1}\times...,I_{N}$. Element-wise, we have
	$
	\label{a}
	(\mathcal{A}\times_n\textbf{U})_{i_{1,...,i_{n-1}ji_{n+1},...,i_{N}}}= \sum_{i_{n}}^{I_n}\mathbf{a}_{i_1i_2...i_{N}}u_{ji_{n}}.	 
	$
	The mode-$n$ product can also be represented as matrix multiplication:
	$
	\mathcal{Y}=\mathcal{A}\times_n\textbf{U} \iff  \textbf{Y}_{(n)}=\textbf{U}\mathcal{A}_{(n)}.
	$
	The result of a series of multiplications in different modes is independent of the order of multiplication, i.e.,
	$
	\mathcal{A}\times_{n}\textbf{U}\times_{m}\textbf{V}=\mathcal{A}\times_{m}\textbf{V}\times_{n}\textbf{U} \quad(m\neq n).
	$
	If the modes are the same, then
	$
	\mathcal{A}\times_{n}\textbf{U}\times_{n}\textbf{V}=\mathcal{A}\times_{n}(\textbf{VU}).
	$
	Next, we introduce some classical tensor decomposition methods. 
	\begin{defn}\label{B2} \textbf{(CP decomposition \cite{Kolda})}
		Suppose that  $\mathcal{X}\in \Re^{I_1 \times I_2 \times...\times I_N}$. Let $\textbf{A}_{n}  \in \Re^{I_n \times r}$, $n = 1,2,...,N$ be the factor matrices. If
		\begin{equation}\label{BB1}
			\mathcal{X}_{i_1i_2...i_N} = \sum_{p=1}^{r} (\textbf{A}_1)_{i_1p} (\textbf{A}_2)_{i_2p} ... (\textbf{A}_N)_{i_Np}
		\end{equation}
		for $i_n = 1, \ldots, I_n$ and $n = 1, \ldots, N$, then we say that $\mathcal{X}$ has a CP decomposition $\mathcal{X} = [[\textbf{A}_1, \textbf{A}_2,..., \textbf{A}_N]]$. The smallest integer $r$ such that (\ref{BB1}) holds is called the CP rank of $\mathcal{X}$, and is denoted as $\mathrm{CPRank}(\mathcal{X}) = r$.
	\end{defn}
	
	\begin{defn}\label{B3} \textbf{(Tucker rank \cite{Kolda})}
		Suppose that the tensor $\mathcal{X}\in \Re^{I_1 \times I_2 \times...\times I_N}$. We may unfold $\mathcal{X}$ to a matrix $\mathcal{X}_{(n)} \in \Re^{I_n \times I_1...I_{n-1}I_{n+1}...I_N}$. Denote the matrix ranks of $\mathcal{X}_{(n)}$ as $r_n$, $n = 1,2,...,N$, respectively. Then the vector $(r_1, r_2,..., r_N)$ is called the Tucker rank of $\mathcal{X}$, and is denoted as $\mathrm{TuckRank}(\mathcal{X}) = (r_1, r_2,..., r_N)$.
	\end{defn}
	
	\begin{defn}\textbf{(Tucker decomposition \cite{Kolda})}
		Suppose that the tensor $\mathcal{X}\in \Re^{I_1 \times I_2 \times...\times I_N}$. Let $\textbf{U}_n \in \Re^{I_n \times r_n}$, $n = 1,2,...,N$, and $\mathcal{C}  \in \Re^{r_1 \times r_2 \times... \times r_N}$.  If
		\begin{equation}\label{B4}
			\mathcal{X}_{i_1i_2...i_N} = \sum_{p_1=1}^{r_1} \sum_{p_2=1}^{r_2}... \sum_{p_N=1}^{r_N} (\textbf{U}_1)_{i_1p_1} (\textbf{U}_2)_{i_2p_2}... (\textbf{U}_N)_{i_Np_N} \mathcal{C}_{p_1p_2...p_N}
		\end{equation}
		for $i_n = 1, \ldots, I_n$ and $n = 1, \ldots, N$, then we denote $\mathcal{X}$ has a Tucker decomposition $\mathcal{X} = [[\mathcal{C};~ \textbf{U}_1, \textbf{U}_2,..., \textbf{U}_N]]$. The matrices $\textbf{U}_n$, $n = 1,2,...,N$ are called factor matrices of the Tucker decomposition, and  $\mathcal{C}$ is called the Tucker core tensor. We may also denote the Tucker decomposition as
		$
		\mathcal{X} = \mathcal{D} \times_1 \textbf{U}_1 \times_2 \textbf{U}_2 \times...\times_N \textbf{U}_N
		$.
		The Tucker ranks $r_n$, $n = 1,2,..,N$  are the smallest integers such that (\ref{B4}) holds.
	\end{defn} 
	\begin{defn}\textbf{(Triple decomposition \cite{TriB})\label{B5}} 
		Let $\mathcal{X}=(x_{ijt})\in \Re^{n_1\times n_2\times n_3}$ be a nonzero tensor, then $\mathcal{X}$ is said the triple product of a third-order horizontally square tensor $\mathcal{A}=(a_{ijt})\in \Re^{n_1\times r\times r}$, a third-order laterally square tensor $\mathcal{B}=(b_{pjs})\in \Re^{r\times n_2\times r}$, and a third-order frontally square tensor $\mathcal{C}=(b_{pqt})\in \Re^{r\times r\times n_3}$ if 
		\begin{equation}\label{B7}
			\begin{array}{l}
				\mathcal{X}_{ijt}=\sum\limits_{p=1}^r\sum\limits_{q=1}^r\sum\limits_{s=1}^ra_{iqs}b_{pjs}c_{pqt},
			\end{array}
		\end{equation} 
		for $i = 1,\dots , n_1$, $j = 1, \dots , n_2$ and  $t = 1, \dots , n_3$.	For convenience, $\mathcal{X}$ is denoted as
		\begin{equation}\label{B6}
			\begin{array}{l}
				\mathcal{X}=[\mathcal{A}\mathcal{B}\mathcal{C}],
			\end{array}
		\end{equation} 
		and $\mathcal{A}$, $\mathcal{B}$, $\mathcal{C}$ are designated as factor tensors of $\mathcal{X}$. The smallest value of $r$ such that (\ref{B7}) holds is called the triple rank of $\mathcal{X}$, and is denoted as $\mathrm{TriRank}(\mathcal{X})$.  
	\end{defn}  
	

	\section{Implicit Multiple Tensor Decomposition}\label{D}
	In this section, we first present the definition of Multiple decomposition and explore its relationship with other classical tensor decompositions. Then, we  introduce the Implicit Multiple Tensor Decomposition (IMTD), along with  its associated properties.
	
	\subsection{Multiple  decomposition}
	Before formally introducing the proposed Multiple  decomposition, we introduce the concept of generalized triple  decomposition, which is a generalization of triple decomposition to arbitrary-order tensors.
	
	\begin{defn} \textbf{(Generalized triple  decomposition)}\label{Def1}
		Let $\mathcal{X}\in \Re^{I_1\times I_2\times...\times I_N}$ be a nonzero tensor, then $\mathcal{X}$ is said the generalized triple product of  $N$-order  tensors $\mathcal{A}_1\in \Re^{I_1\times r\times...\times r}$, $\mathcal{A}_2\in \Re^{r\times I_2\times r \times...\times r}$, ..., and $\mathcal{A}_N\in \Re^{r\times r\times...\times r\times I_N}$ if 
		\begin{equation}\label{B9}
			\begin{array}{l}
				\mathcal{X}_{i_1i_2...i_N}=\sum\limits_{p_1=1}^r\sum\limits_{p_2=1}^r...\sum\limits_{p_N=1}^r (\mathcal{A}_1)_{i_1p_2...p_N}(\mathcal{A}_2)_{p_1i_2p_3...p_N}...(\mathcal{A}_N)_{p_1p_2...p_{N-1}i_N},
			\end{array}
		\end{equation} 
		for $i_n = 1, \ldots, I_n$ and $n = 1, \ldots, N$.	For convenience, $\mathcal{X}$ is denoted as
		\begin{equation}\label{B99}
			\begin{array}{l}
				\mathcal{X}=[\mathcal{A}_1\mathcal{A}_2...\mathcal{A}_N],
			\end{array}
		\end{equation} 
		and $\mathcal{A}_1$, $\mathcal{A}_2$,..., $\mathcal{A}_N$ are designated as factor tensors of $\mathcal{X}$. The smallest  $r$ such that (\ref{B9}) holds is called the generalized triple rank of $\mathcal{X}$, and is denoted as $\mathrm{GTriRank}(\mathcal{X})$. For a zero tensor, its generalized triple rank is defined as zero. 
	\end{defn}  
	
	Next, the well-definedness of Definition \ref{Def1} and an upper bound on the generalized triple rank are established in the following theorem.

	\begin{theorem}
		Generalized triple  decomposition and generalized triple rank are well defined. A  nonzero tensor $\mathcal{X}\in \Re^{I_1\times I_2\times...\times I_N}$ always has a generalized triple  decomposition (\ref{B9}), satisfying $
			\mathrm{GTriRank}(\mathcal{X})\le \mathrm{submax}\{I_1, I_2,...,I_N\}$.  
	\end{theorem}
	\begin{proof}
	Without loss of generality, we may assume that we have a  nonzero tensor $\mathcal{X}\in \Re^{I_1\times I_2\times...\times I_N}$ and $I_1 \geq I_2\geq... \geq I_N \geq 1$. Thus, $\mathrm{submax}\{I_1, I_2,..., I_N\} = I_2$. Let $r = I_2$. Define tensors $\mathcal{A}_1 \in \Re^{I_1 \times r \times...\times r}$, $\mathcal{A}_2 \in \Re^{r \times I_2 \times...\times r}$,..., and $\mathcal{A}_N \in \Re^{r \times r\times... \times I_N}$,  such that 
	$(\mathcal{A}_1)_{i_1p_2...p_N} = x_{i_1p_Np_{N-1}...p_3p_2}$ for $ p_3\le I_3, p_4\le I_4,...,p_N\le I_N $ and 
	$(\mathcal{A}_1)_{i_1p_2...p_N} = 0$  otherwise,
	$(\mathcal{A}_2)_{1i_2...p_N} = \delta_{i_2p_3...p_N}$ and $(\mathcal{A}_2)_{p_1i_2...p_N} = 0$ for $p_1 > 1$, ...,  $(\mathcal{A}_N)_{1p_2...i_N} = \delta_{p_2...p_{N-1}i_N}$ and $(\mathcal{A}_N)_{p_1p_2...i_N} = 0$ for $p_1 > 1$,  where $\delta_{p_2...p_{N-1}i_N}$,..., and $\delta_{p_2...p_{N-1}i_N}$ are the Kronecker symbol such that $\delta=1$  when all  subscripts are equal, and $0$ otherwise. Then, we get
	\begin{eqnarray*}
		&\sum\limits_{p_1=1}^r\sum\limits_{p_2=1}^r...\sum\limits_{p_N=1}^r(\mathcal{A}_1)_{i_1p_2...p_N}(\mathcal{A}_2)_{p_1i_2...p_N}...(\mathcal{A}_N)_{p_1p_2...i_N} \\
		= &\sum_{p_2=1}^{r}... \sum_{p_N=1}^{r} x_{i_1p_N...p_2} \delta_{i_2p_{3}...p_N}... \delta_{p_2...p_{N-1}i_N} = \mathcal{X}_{i_1i_2...i_N},
	\end{eqnarray*}
	for $i_n = 1, \ldots, I_n$ and $n = 1, \ldots, N$, i.e., (3.1) holds for the above choices of factor tensors $\mathcal{A}_1$, $\mathcal{A}_2$,..., and $\mathcal{A}_N$. Thus, generalized triple decomposition always exists with GTriRank($\mathcal{X}$)$ \leq r:=I_2=\mathrm{submax}\{I_1, I_2,...,I_N\}$.
\end{proof}

	Although the generalized triple decomposition is well-defined, it requires the short dimensions of the factor tensors to be identical, which limits its flexibility and applicability. Therefore, we proceed to introduce the proposed Multiple  decomposition.

	\begin{defn} \textbf{(Multiple decomposition)}
		Let  tensor $\mathcal{X}\in \Re^{I_1\times I_2\times...\times I_N}$ be a nonzero tensor, then $\mathcal{X}$ is said the  Multiple product of  $N$-order  tensors $\mathcal{A}_1\in \Re^{I_1\times r_2\times...\times r_N}$, $\mathcal{A}_2\in \Re^{r_1\times I_2\times r_3\times...\times r_N}$, ..., and $\mathcal{A}_N\in \Re^{r_1\times r_2\times...\times r_{N-1}\times I_N}$ if 
		\begin{equation}\label{B_9}
			\begin{array}{l}
				\mathcal{X}_{i_1i_2...i_N}=\sum\limits_{p_1=1}^{r_1}\sum\limits_{p_2=1}^{r_2}...\sum\limits_{p_N=1}^{r_N} (\mathcal{A}_1)_{i_1p_2...p_N}(\mathcal{A}_2)_{p_1i_2p_3...p_N}...(\mathcal{A}_N)_{p_1p_2...i_{N-1}i_N},
			\end{array}
		\end{equation} 
		for $i_n = 1, \ldots, I_n$ and $n = 1, \ldots, N$.	For convenience, we still use  (\ref{B99}) to represent $\mathcal{X}$, i.e., $
		\mathcal{X}=[\mathcal{A}_1\mathcal{A}_2...\mathcal{A}_N] $,
		and $\mathcal{A}_1$, $\mathcal{A}_2$, ..., $\mathcal{A}_N$ are designated as  factor tensors of $\mathcal{X}$. 
		If the vector $\textbf{r} := (r_1, r_2, \ldots, r_N)$ satisfies the following conditions:
		\begin{enumerate}
			\item $r_i \leq \mathrm{GTriRank}(\mathcal{X})$ for all $i = 1, 2, \ldots, N$; 
			
			\item $\textbf{r}$ has the smallest $\ell_1$ norm among all such vectors satisfying (\ref{B_9}). In the case of Multiple vectors sharing the same minimal $\ell_1$ norm, $\textbf{r}$ is selected to be the one with the smallest lexicographical order;
		\end{enumerate}
		then we refer to  $\textbf{r}$ as the \emph{ Multiple rank} of $\mathcal{X}$, and denote it by $\mathrm{MtpRank}(\mathcal{X})$.
	\end{defn}
	
	The  Multiple rank of tensor $\mathcal{X}$ always exists, since we can at least choose $ r_i = {\rm GTriRank(\mathcal{X})} $ for $i = 1,2,...,N$. Note that if the vector $(r_1, r_2, \ldots, r_N)$ satisfies (\ref{B_9}), then any vector $(s_1, s_2, \ldots, s_N)$ with $s_i \ge r_i$ can also satisfy equation (\ref{B_9}), simply by zero-padding the corresponding factor tensors. Therefore,  condition 2 is  imposed to ensure the uniqueness of the Multiple rank. 	The following theorem establishes the relationship between GTriRank($\mathcal{X}$) and MtpRank($\mathcal{X}$).
    \newpage
	\begin{theorem}\label{T11}
		If $\mathcal{X}\in \mathbb{R}^{I_1\times I_2\times\cdots\times I_N}$ is a nonzero tensor with $\mathrm{MtpRank}(\mathcal{X})=(r_1, r_2, \ldots, r_N)$, then $\mathrm{GTriRank}(\mathcal{X})=\max\{r_1,r_2,\ldots,r_N\}$.
	\end{theorem}
	\begin{proof}
		Due to  MtpRank($\mathcal{X}$)= $(r_1,r_2,...,r_N)$, there exists $\mathcal{A}_1\in \Re^{I_1\times r_2\times...\times r_N}$, $\mathcal{A}_2\in \Re^{r_1\times I_2\times...\times r_N}$, ..., and $\mathcal{A}_N\in \Re^{r_1\times r_2\times...\times r_{N-1}\times I_N}$, satisfying $\mathcal{X}=[\mathcal{A}_1\mathcal{A}_2...\mathcal{A}_N]$. Let $r:=\max\{r_1,r_2,...,r_N\}$. 
		We expand $\mathcal{A}_1$, $\mathcal{A}_2$, \ldots, and $\mathcal{A}_N$ to $\widehat{\mathcal{A}}_1 \in \mathbb{R}^{I_1 \times r \times \cdots \times r}$, $\widehat{\mathcal{A}}_2 \in \mathbb{R}^{r \times I_2 \times \cdots \times r}$, \ldots, and $\widehat{\mathcal{A}}_N \in \mathbb{R}^{r \times r \times \cdots \times r \times I_N}$ by zero-padding. Then, we have 
		\begin{equation*}
			\begin{aligned}
				&\quad~\sum\limits_{p_1=1}^r\sum\limits_{p_2=1}^r...\sum\limits_{p_N=1}^r(\widehat{\mathcal{A}}_1)_{i_1p_2...p_N}(\widehat{\mathcal{A}}_2)_{p_1i_2...p_N}...(\widehat{\mathcal{A}}_N)_{p_1p_2...i_N} \\&=\sum\limits_{p_1=1}^{r_1}\sum\limits_{p_2=1}^{r_2}...\sum\limits_{p_N=1}^{r_N} (\mathcal{A}_1)_{i_1p_2...p_N}(\mathcal{A}_2)_{p_1i_2...p_N}...(\mathcal{A}_N)_{p_1p_2...i_N} = \mathcal{X}_{i_1i_2...i_N}.	
			\end{aligned}
		\end{equation*}
		Therefore, we also have $\mathcal{X}=[\widehat{\mathcal{A}}_1\widehat{\mathcal{A}}_2...\widehat{\mathcal{A}}_N]$, which implies ${\rm GTriRank(\mathcal{X})}\le r$. According to the definition of  Multiple rank, we  obtain ${\rm GTriRank(\mathcal{X})}\ge r $, which means ${\rm GTriRank(\mathcal{X})}=\max\{r_1,r_2,...,r_N\}$.
	\end{proof}
	
	\begin{remark}
		Theorem~\ref{T11} implies that the requirement of equal auxiliary dimensions in the factor tensors may introduce redundancy, as they are uniformly set to $\max\{r_1, \ldots, r_N\}$. 
		Therefore, Multiple decomposition, by allowing flexible rank configurations, can achieve a higher compression ratio and improved efficiency.
	\end{remark}

	Next, we introduce some operational properties of the Multiple product.
	\begin{theorem}\label{T22}
		Suppose that tensor $\mathcal{A}_1\in \Re^{I_1\times r_2\times...\times r_N}$, $\mathcal{A}_2\in \Re^{r_1\times I_2\times...\times r_N}$, ...,  $\mathcal{A}_N\in \Re^{r_1\times r_2\times...\times r_{N-1}\times I_N}$ and $\widehat{\mathcal{A}}_1\in \Re^{I_1\times r_2\times...\times r_N}$, $\widehat{\mathcal{A}}_2\in \Re^{r_1\times I_2\times...\times r_N}$, ...,  $\widehat{\mathcal{A}}_N\in \Re^{r_1\times r_2\times...\times I_N}$ are N-order  tensors, then the following properties hold:
		\begin{enumerate}
			\item $[(\mathcal{A}_1-\widehat{\mathcal{A}}_1)\mathcal{A}_2...\mathcal{A}_N]=[\mathcal{A}_1\mathcal{A}_2...\mathcal{A}_N]-[\widehat{\mathcal{A}}_1\mathcal{A}_2...\mathcal{A}_N]$,
			
			$[\mathcal{A}_1(\mathcal{A}_2-\widehat{\mathcal{A}}_2)...\mathcal{A}_N]=[\mathcal{A}_1\mathcal{A}_2...\mathcal{A}_N]-[\mathcal{A}_1\widehat{\mathcal{A}}_2...\mathcal{A}_N]$,
			
			\qquad ...
			
			$[\mathcal{A}_1\mathcal{A}_2...(\mathcal{A}_N-\widehat{\mathcal{A}}_N)]=[\mathcal{A}_1\mathcal{A}_2...\mathcal{A}_N]-[\mathcal{A}_1\mathcal{A}_2...\widehat{\mathcal{A}}_N]$.
			
			\item If $\mathcal{X}=[\mathcal{A}_1\mathcal{A}_2...\mathcal{A}_N]$, then we have 
			$$
			|\mathcal{X}_{i_1i_2...i_N}|\le \Vert \mathcal{A}_1(i_1,:,...,:) \Vert_{\ell_1}\Vert \mathcal{A}_2(:,i_2,...,:) \Vert_{\ell_1}...\Vert \mathcal{A}_N(:,:,...,i_N) \Vert_{\ell_1}.
			$$
			
			\item If $\mathcal{X}=[\mathcal{A}_1\mathcal{A}_2...\mathcal{A}_N]$, then for $n = 1,2,...,N$, one has 
			$$
			\mathcal{X}_{(n)} = (\mathcal{A}_n)_{(n)}\mathcal{E}_{n[N-1]}^\top,
			$$
			where $\mathcal{E}_n$ is a $(2N-2)$-order tensor with size $I_1\times...\times I_{n-1}\times I_{n+1}\times...\times I_N\times r_1 \times \ldots \times r_{n-1}\times r_{n+1}\times...\times r_{N}$, and
			$
			{(\mathcal{E}_n)_{i_1...i_{n-1}i_{n+1}...i_Np_1...p_{n-1}p_{n+1}...p_N}}=\sum\limits_{p_n=1}^{r_n}(\mathcal{A}_1)_{i_1p_2...p_N}...(\mathcal{A}_{n-1})_{p_1...i_{n-1}p_n...p_N}$ $(\mathcal{A}_{n+1})_{p_1...p_{n}i_{n+1}...p_N}...(\mathcal{A}_N)_{p_1p_2...i_N}$.
			
			The $\mathcal{E}_{n[N-1]}$  is a  matrix by collapsing the first $N-1$ dimensions of $\mathcal{E}$ and reshaping the remaining dimensions into a single dimension. 
		\end{enumerate}
	\end{theorem}
	\begin{proof}
	Properties 1 and 2 follow directly from the definition of multiple product. We focus on proving Property 3.  
	To clarify the procedure, we present the mode-1 unfolding of the tensor as an illustrative example, noting that the unfoldings along other dimensions follow analogously:
	\begin{equation}
		\begin{aligned}
			\mathcal{X}_{i_1i_2...i_N}&=\sum\limits_{p_1=1}^{r_1}\sum\limits_{p_2=1}^{r_2}...\sum\limits_{p_N=1}^{r_N} (\mathcal{A}_1)_{i_1p_2...p_N}(\mathcal{A}_2)_{p_1i_2...p_N}...(\mathcal{A}_N)_{p_1p_2...i_N}\\
			& =\sum\limits_{p_2=1}^{r_2}...\sum\limits_{p_N=1}^{r_N} (\mathcal{A}_1)_{i_1p_2...p_N}(\mathcal{A}_2)_{1i_2...p_N}...(\mathcal{A}_N)_{1p_2...i_N} \\
			&+\sum\limits_{p_2=1}^{r_2}...\sum\limits_{p_N=1}^{r_N} (\mathcal{A}_1)_{i_1p_2...p_N}(\mathcal{A}_2)_{2i_2...p_N}...(\mathcal{A}_N)_{2p_2...i_N} \\
			& \qquad...\\
			&+\sum\limits_{p_2=1}^{r_2}...\sum\limits_{p_N=1}^{r_N} (\mathcal{A}_1)_{i_1p_2...p_N}(\mathcal{A}_2)_{r_1i_2...p_N}...(\mathcal{A}_N)_{r_1p_2...i_N}.
		\end{aligned}
	\end{equation}
	Notice that each term shares the common factor $(\mathcal{A}_1)_{i_1p_2...p_N}$. Therefore, one has
	\begin{equation}\label{E3}
		\begin{aligned}
			\mathcal{X}_{i_1i_2...i_N}=\sum\limits_{p_2=1}^{r_2}...\sum\limits_{p_N=1}^{r_N}[ (\mathcal{A}_1)_{i_1p_2...p_N}\textbf{e}_{i_2...i_Np_2...p_N}],
		\end{aligned}
	\end{equation}
	where $\textbf{e}_{i_2...i_Np_2...p_N}:=\sum\limits_{p_1=1}^{r_1}(\mathcal{A}_2)_{p_1i_2p_3...p_N}...(\mathcal{A}_N)_{p_1p_2...p_{N-1}i_N}$. We can rewrite (\ref{E3}) as:
	\begin{equation}\label{E4}
		\begin{aligned}
			\mathcal{X}_{i_1i_2...i_N}=\sum\limits_{k=1}^{K}[ (\mathcal{A}_1)_{i_1k}\textbf{e}_{i_2...i_Nk}],\quad \text{with}\quad  K = r_2r_3...r_N.
		\end{aligned}
	\end{equation}
	The right-hand side of equation (\ref{E4}) can be viewed as a matrix multiplication:
	\begin{equation}\label{E5}
		\begin{aligned}
			\mathcal{X}_{i_1i_2...i_N}= [(\mathcal{A}_1)_{(1)}\mathcal{E}_{[N-1]}^\top]_{i_1(i_2...i_N)}.
		\end{aligned}
	\end{equation}
	Since $\mathcal{X}_{i_1i_2...i_N}$ corresponds to $(\mathcal{X}_{(1)})_{i_1(i_2...i_N)}$, we can obtain
	$$
	\mathcal{X}_{(1)}=(\mathcal{A}_1)_{(1)}\mathcal{E}_{[N-1]}^\top.
	$$
	The computation of other mode unfoldings can be derived similarly to that of the first mode, which completes the proof.
\end{proof}
	
	\begin{remark}
		Property 1 establishes that the Multiple product satisfies the distributive law, while Property 2 provides a useful upper bound for this operation. Both  follow directly from the definition of  Multiple product.  Property 3 presents a computational formula for unfolding the Multiple product of factor tensors across different modes, which can be used to solve the Multiple decomposition by iteratively updating the factor tensors within the block coordinate descent (BCD) framework.
	\end{remark}

	\subsubsection{The relation between the Multiple  and CP decompositions}
	In fact,  Multiple decomposition can be  viewed as an extension of  CP decomposition, where each factor matrix in  CP decomposition is generalized to a factor tensor in Multiple decomposition, as illustrated in Figure \ref{CF1}. 
	\begin{figure}
		\centering
		\subfloat{\includegraphics[width=6.2cm]{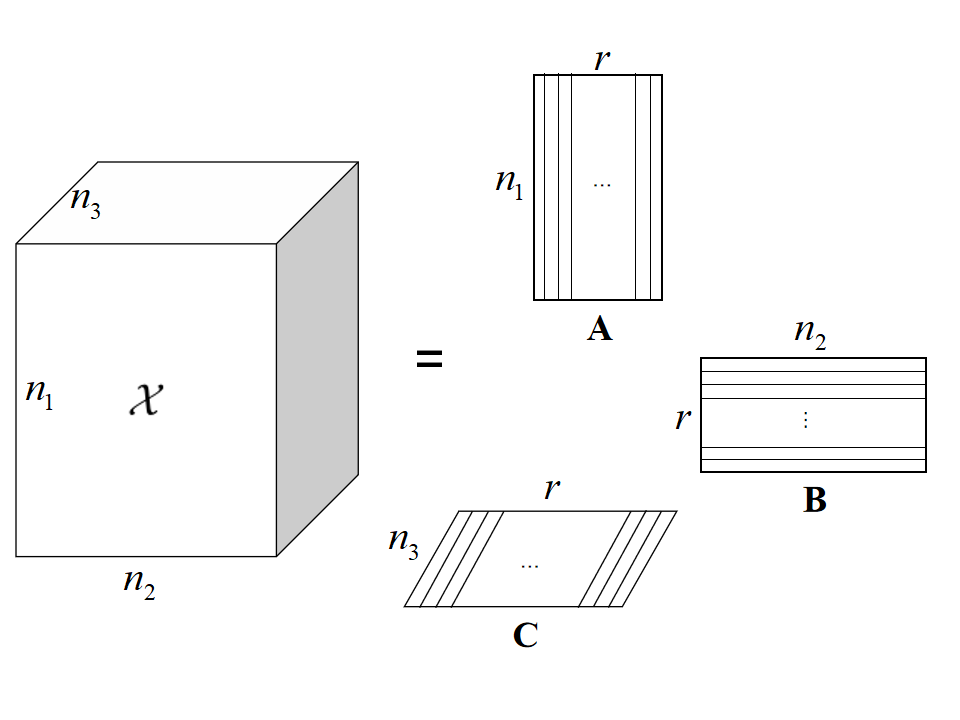}}\quad
		\subfloat{\includegraphics[width=6.2cm]{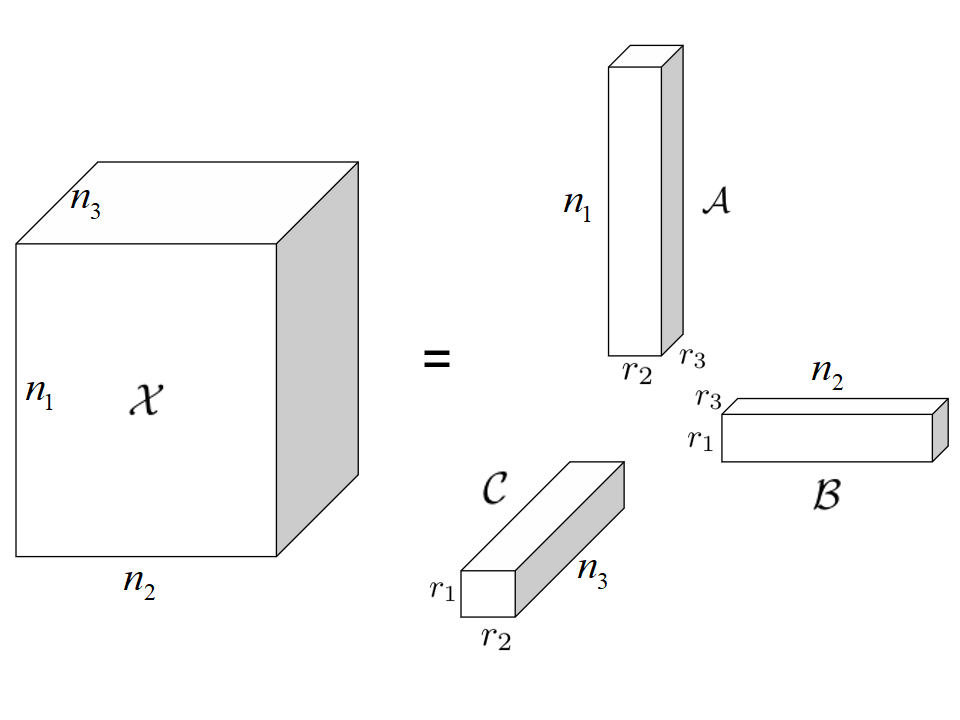}}
		
		\caption{ \small An intuitive comparison between CP decomposition (left) and Multiple decomposition (right) for a third-order tensor.   }
		\label{CF1}
	\end{figure}
	The formal relationship can be established in the following theorem.
	\begin{theorem}\label{T23}
		For $\mathcal{X}\in \Re^{I_1\times I_2\times...\times I_N}$ with MtpRank($\mathcal{X}$)=$(r_1, r_2,...,r_N)$, we have
		\begin{eqnarray}
			\max\{r_1,r_2,...,r_N\}\le {\rm CPRank(\mathcal{X})} \le r_1r_2...r_N. \label{D3}
		\end{eqnarray}
	\end{theorem}
	\begin{proof}
		We denote $r:={\rm CPRank(\mathcal{X})}$. Suppose  $\mathcal{X} = [[\textbf{A}_1, \textbf{A}_2, ..., \textbf{A}_N]]$ is a CP decomposition of tensor $\mathcal{X}$, where $\textbf{A}_1 = (a^{(1)}_{i_1p}) \in \Re^{I_1 \times r}$, $\textbf{A}_2 = (a^{(2)}_{i_2p}) \in \Re^{I_2 \times r}$,..., and $\textbf{A}_N = (a^{(N)}_{i_np}) \in \Re^{I_N \times r}$ are factor matrices. Denote $\mathcal{A}_1  \in \Re^{I_1 \times r\times... \times r}$, $\mathcal{A}_2  \in \Re^{r \times I_2 \times...\times r}$, ... , and $\mathcal{A}_N  \in \Re^{r \times r \times...\times I_N}$ with
		$(\mathcal{A}_1)_{i_1p_2...p_N} = a^{(1)}_{i_1p_2}$ for $p_2=p_3=...=p_N$ and $(\mathcal{A}_1)_{i_1p_2...p_N} = 0$ otherwise;
		$(\mathcal{A}_2)_{p_1i_2...p_N} = a^{(2)}_{i_2p_1}$ for $ p_1=p_3=...=p_N$ and $(\mathcal{A}_2)_{p_1i_2...p_N} = 0$ otherwise; ... ; $
		(\mathcal{A}_N)_{p_1p_2...I_N} = 	a^{(N)}_{i_Np_1} $ for $ p_1=p_2=...=p_{N-1}$.
		Then for all $i_n = 1, \ldots, I_n$ and $n = 1, \ldots, N$, there holds
		\begin{equation*}
			\begin{aligned}
				[\mathcal{A}_1\mathcal{A}_2...\mathcal{A}_N]_{i_1i_2...i_N} &= \sum\limits_{p_1=1}^r\sum\limits_{p_2=1}^r...\sum\limits_{p_N=1}^r(\mathcal{A}_1)_{i_1p_2...p_N}(\mathcal{A}_2)_{p_1i_2...p_N}...(\mathcal{A}_N)_{p_1p_2...i_N} \\
				&= \sum_{p=1}^{r} a^{(1)}_{i_1p} a^{(2)}_{i_2p}... a^{(N)}_{i_Np} = \mathcal{X}_{i_1i_2...i_N}.
			\end{aligned}
		\end{equation*}
		This means $\mathcal{X} = [\mathcal{A}_1\mathcal{A}_2...\mathcal{A}_N]$ and $\mathrm{GTriRank}(\mathcal{X}) \leq \mathrm{CPRank}(\mathcal{X})$. According to Theorem \ref{T11}, we have $\max\{r_1,r_2,...,r_N\}\le {\rm CPRank(\mathcal{X})}$.
		On the other hand, since MtpRank($\mathcal{X}$)=$(r_1, r_2,...,r_N)$, There exists $\mathcal{B}_1\in \Re^{I_1\times r_2\times...\times r_N}$, $\mathcal{B}_2\in \Re^{r_1\times I_2\times...\times r_N}$, ..., and  $\mathcal{B}_N\in \Re^{r_1\times r_2\times...\times r_{N-1}\times I_N}$, s.t.
		\begin{equation*}
			\mathcal{X}_{i_1i_2...i_N} = \sum_{p_1=1}^{r_1} \sum_{p_2=1}^{r_2}... \sum_{p_N=1}^{r_N} (\mathcal{B}_1)_{i_1p_2...p_N} (\mathcal{B}_2)_{p_1i_2...p_N}... (\mathcal{B}_N)_{p_1p_2...i_N}.
		\end{equation*} 
		Therefore, $\mathcal{X}$ can be represented as a sum of $r_1r_2...r_N$ rank-one tensors, and the last inequality in the theorem holds.
	\end{proof} 
	
	The second inequality in (\ref{D3}) can hold with equality (see Example 3.5 in \cite{TriB}), in which case the CP rank can be much larger than the Multiple rank.

	\subsubsection{The relation between the Multiple  and  Tucker decompositions}
	If tensor $\mathcal{X} = [\mathcal{A}_1\mathcal{A}_2...\mathcal{A}_N]$ with $\mathcal{A}_1 = \mathcal{F}_1 \times_1 \textbf{U}_1$, $\mathcal{F}_1 \in \Re^{R_1 \times r_2\times... \times r_N}$, $\textbf{U}_1 \in \Re^{I_1 \times R_1}$, $\mathcal{A}_2 = \mathcal{F}_2 \times_2 \textbf{U}_2$, $\mathcal{F}_2 \in \Re^{r_1 \times R_2\times... \times r_N}$, $\textbf{U}_2 \in \Re^{I_2 \times R_2}$, ..., and $\mathcal{A}_N = \mathcal{F}_N \times_N \textbf{U}_N$, $\mathcal{F}_N \in \Re^{r_1 \times r_2\times... \times R_N}$,  $\textbf{U}_N \in \Re^{I_N \times R_N}$, then we have
	\begin{equation}\label{S10}
		\begin{aligned}
			\mathcal{X}_{i_1i_2...i_N} &= \sum_{p_1=1}^{r_1} \sum_{p_2=1}^{r_2}... \sum_{p_N=1}^{r_N} (\mathcal{A}_1)_{i_1p_2...p_N} (\mathcal{A}_2)_{p_1i_2...p_N}...(\mathcal{A}_N) _{p_1p_2...p_{N-1}i_N} \\&= \sum_{j_1=1}^{R_1} \sum_{j_2=1}^{R_2}... \sum_{j_N=1}^{R_N} (\textbf{U}_1)_{i_1j_1} (\textbf{U}_2)_{i_2j_2}... (\textbf{U}_N)_{i_Nj_N} \\&~ \cdot\underbrace{\sum_{p_1=1}^{r_1} \sum_{p_2=1}^{r_2}... \sum_{p_N=1}^{r_N} (\mathcal{F}_1)_{j_1p_2...p_N} (\mathcal{F}_2)_{p_1j_2...p_N}... (\mathcal{F}_N)_{p_1p_2...p_{N-1}j_N}}_{ ([\mathcal{F}_1\mathcal{F}_2...\mathcal{F}_N])_{j_1j_2...j_N}},
		\end{aligned}
	\end{equation}
	which corresponds to the formulation in (\ref{B4}). This implies that
	\begin{equation}\label{C36}
		[( \mathcal{F}_1 \times_1 \textbf{U}_1 )( \mathcal{F}_2 \times_2 \textbf{U}_2 )...(\mathcal{F}_N \times_N \textbf{U}_N )]=[ \mathcal{F}_1 \mathcal{F}_2... \mathcal{F}_N ] \times_1 \textbf{U}_1 \times_2 \textbf{U}_2... \times_n \textbf{U}_N.
	\end{equation}
	Based on (\ref{C36}), we can derive the following conclusions.
	\begin{theorem}\label{T55}
		Suppose  $\mathcal{X}\in \Re^{I_1\times I_2\times...\times I_N}$ 
		with ${\rm TuckRank(\mathcal{X})}=(t_1,t_2,...,t_N)$. Let
		$
		\mathcal{X} = \mathcal{D} \times_1 \textbf{U}_1 \times_2 \textbf{U}_2 \times... \times_N \textbf{U}_N
		$
		be a Tucker decomposition of $\mathcal{X}$ with $\mathcal{D} \in \Re^{t_1 \times t_2 \times...\times t_N}$ and factor matrices $\textbf{U}_1 \in \Re^{I_1 \times t_1}$, $\textbf{U}_2 \in \Re^{I_2 \times t_2}$,..., $\textbf{U}_N \in \Re^{I_N \times t_N}$. If $ \textbf{U}_i^\top \textbf{U}_i $, $ i = 1,2,...,N $ are invertible, then one has
		\begin{eqnarray}
			\mathrm{MtpRank} (\mathcal{X}) = \mathrm{MtpRank}(\mathcal{D}) \leq \mathrm{submax}\{ t_1, t_2,..., t_N \}.
		\end{eqnarray}
	\end{theorem}
	\begin{proof}  Assume 
		$
		\mathcal{D} =  [\widetilde{\mathcal{A}}_1 \widetilde{\mathcal{A}}_2... \widetilde{\mathcal{A}}_N ]
		$
		with $\widetilde{\mathcal{A}}_1 \in \Re^{t_1 \times r_2\times... \times r_N}$,
		$\widetilde{\mathcal{A}}_2 \in \Re^{r_1 \times t_2\times... \times r_N}$, ..., and
		$\widetilde{\mathcal{A}}_N \in \Re^{r_1 \times r_2 \times...\times t_N}$. Then from (\ref{C36}), one has
		\begin{equation}\label{C2}
			\begin{aligned}
				\hspace{-0cm}\mathcal{X} &= [ \widetilde{\mathcal{A}}_1 \widetilde{\mathcal{A}}_2... \widetilde{\mathcal{A}}_N ] \times_1 \textbf{U}_1 \times_2 \textbf{U}_2... \times_n \textbf{U}_N\\ &
				= [( \widetilde{\mathcal{A}}_1 \times_1 \textbf{U}_1 )( \widetilde{\mathcal{A}}_2 \times_2 \textbf{U}_2 )...( \widetilde{\mathcal{A}}_N \times_N \textbf{U}_N )].
			\end{aligned}
		\end{equation}
		Clearly, $\widetilde{\mathcal{A}}_1 \times_1 \textbf{U}_1 \in \Re^{I_1 \times r_2\times... \times r_N}$,
		$\widetilde{\mathcal{A}}_2 \times_2 \textbf{U}_2 \in \Re^{r_1 \times I_2\times... \times r_N}$,..., and
		$\widetilde{\mathcal{A}}_N \times_N \textbf{U}_N \in \Re^{r_1 \times r_2 \times...\times I_N}$. Since  $ \textbf{U}_i^\top \textbf{U}_i $, $ i = 1,2,...,N $ are invertible, one has
		\begin{equation*}
			\begin{aligned}
				&\quad~ \mathcal{X} \times_1  ( \textbf{U}_1^\top \textbf{U}_1  )^{-1} \textbf{U}_1^\top \times_2  ( \textbf{U}_2^\top \textbf{U}_2  )^{-1} \textbf{U}_2^\top\times... \times_N  ( \textbf{U}_N^\top \textbf{U}_N  )^{-1} \textbf{U}_N^\top \\
				& = \mathcal{D} \times_1  ( \textbf{U}_1^\top \textbf{U}_1  )^{-1} (\textbf{U}_1^\top \textbf{U}_1) \times_2  ( \textbf{U}_2^\top \textbf{U}_2  )^{-1} (\textbf{U}_2^\top \textbf{U}_2) \times...\times_N  ( \textbf{U}_N^\top \textbf{U}_N  )^{-1} (\textbf{U}_N^\top \textbf{U}_N) \\
				&=  \mathcal{D} \times_1 \textbf{I}_{r_1} \times_2 \textbf{I}_{r_2}\times... \times_N \textbf{I}_{r_N} =  \mathcal{D}.
			\end{aligned}
		\end{equation*}
		Assume that ${\rm MtpRank(\mathcal{X})}=(R_1, R_2, ..., R_N)$, then there  exists $\mathcal{A}_1\in \Re^{I_1\times R_2\times...\times R_N}$, $\mathcal{A}_2\in \Re^{R_1\times I_2\times...\times R_N}$, ...,  and $\mathcal{A}_N\in \Re^{R_1\times R_2\times...\times R_{N-1}\times I_N}$, s.t., $\mathcal{X}=[\mathcal{A}_1\mathcal{A}_2...\mathcal{A}_N]$.
		Hence, it holds that
		\begin{equation*}\label{C22}
			\begin{aligned}
				\mathcal{D} &=   [\mathcal{A}_1 \mathcal{A}_2... \mathcal{A}_N ]\times_1  ( \textbf{U}_1^\top \textbf{U}_1  )^{-1} \textbf{U}_1^\top \times_2  ( \textbf{U}_2^\top \textbf{U}_2  )^{-1} \textbf{U}_2^\top \times...\times_N  ( \textbf{U}_N^\top \textbf{U}_N  )^{-1} \textbf{U}_N^\top \nonumber \\
				&= [ (  \mathcal{A}_1 \times_1  ( \textbf{U}_1^\top \textbf{U}_1  )^{-1} \textbf{U}_1^\top  )  (  \mathcal{A}_2 \times_2  ( \textbf{U}_2^\top \textbf{U}_2  )^{-1} \textbf{U}_2^\top)... (  \mathcal{A}_N \times_N  ( \textbf{U}_N^\top \textbf{U}_N  )^{-1} \textbf{U}_N^\top  )].
			\end{aligned}
		\end{equation*}
		We can find that
		$
		\mathcal{A}_1 \times_1  ( \textbf{U}_1^\top \textbf{U}_1  )^{-1} \textbf{U}_1^\top  \in \Re^{t_1 \times R_2\times... \times R_N}, ~
		\mathcal{A}_2 \times_2  ( \textbf{U}_2^\top \textbf{U}_2  )^{-1} \textbf{U}_2^\top \in \Re^{R_1 \times t_2\times... \times R_N},
		$
		and
		$
		\mathcal{A}_N \times_N  ( \textbf{U}_N^\top \textbf{U}_N  )^{-1} \textbf{U}_N^\top \in \Re^{R_1 \times R_2 \times...\times t_N}
		$.
		Combined with (\ref{C2}), it implies that the factor tensors of $\mathcal{D}$ and those of $\mathcal{X}$ can be transformed into each other. Therefore,  we have ${\rm MtpRank(\mathcal{X})} = {\rm MtpRank(\mathcal{D})} $. 
	\end{proof}
	
	In fact, there is also a close connection between the tensor ranks of $\mathcal{X}$ and those of its factor tensors in generalized triple  decomposition.
	
	\begin{theorem}\label{T66}
		Let $\mathcal{X} = [\mathcal{A}_1\mathcal{A}_2...\mathcal{A}_N]$ be the generalized triple  decomposition; then for $n = 1, 2, ...,N$, one has
		\begin{equation*}
			\mathrm{TuckRank}(\mathcal{X})_n \leq \mathrm{TuckRank}(\mathcal{A}_n)_n \leq \mathrm{GTriRank}(\mathcal{A}_n)^{N-1} \leq \mathrm{GTriRank}(\mathcal{X})^{N-1}. 
		\end{equation*}
	\end{theorem}
	\begin{proof}
	Suppose  $\mathrm{GTriRank}(\mathcal{X}) = r$,  $\mathrm{TuckRank}(\mathcal{A}_1)_1 = r_1$, $\mathrm{TuckRank}(\mathcal{A}_2)_2 = r_2$,..., and $\mathrm{TuckRank}(\mathcal{A}_N)_N = r_N$. Let $\mathcal{A}_1 = \mathcal{F} \times_1 \textbf{U}_1 \times_2 \textbf{V}_2 \times...\times_N \textbf{V}_n$ be a Tucker decomposition of $\mathcal{A}_1$ with core tensor $\mathcal{F} \in \Re^{r_1 \times s_2 \times...\times s_N}$ and factor matrices $\textbf{U}_1 \in \Re^{I_1 \times r_1}$, $\textbf{V}_2 \in \Re^{r \times s_2}$,..., $\textbf{V}_N \in \Re^{r \times s_N}$. Denote $\bar{\mathcal{A}}_1 = \mathcal{F} \times_2 \textbf{V}_2\times... \times_N \textbf{V}_N \in \Re^{r_1 \times r \times... \times r}$. Then $\mathcal{A}_1 = (\mathcal{F} \times_2 \textbf{V}_2\times... \times_N \textbf{V}_N) \times_1 \textbf{U}_1 = \bar{\mathcal{A}}_1 \times_1 \textbf{U}_1$.
	Similarly, there exist $\bar{\mathcal{A}}_2 \in \Re^{r \times r_2 \times...\times r}$,..., $\bar{\mathcal{A}}_N \in \Re^{r \times r \times...\times r_N}$, and $\textbf{U}_2 \in \Re^{I_2 \times r_2}$,..., $\textbf{U}_N \in \Re^{I_N \times r_N}$ such that
	$
	\mathcal{A}_1 = \bar{\mathcal{A}}_1 \times_1 \textbf{U}_1, \quad \mathcal{A}_2 = \bar{\mathcal{A}}_2 \times_2 \textbf{U}_2, ~...,~ \mathcal{A}_N = \bar{\mathcal{A}}_N \times_N \textbf{U}_N.
	$
	Hence, $\mathcal{X} = [\mathcal{A}_1\mathcal{A}_2...\mathcal{A}_N] = ([\bar{\mathcal{A}_1} \bar{\mathcal{A}_2} ...\bar{\mathcal{A}}_N]) \times_1 \textbf{U}_1 \times_2 \textbf{U}_2\times... \times_N \textbf{U}_N$ according to (3.5). From the definition of Tucker rank, for $n=1, 2, ..., N$, we have
	\begin{equation}\label{E10}
		\mathrm{TuckRank}(\mathcal{X})_n \leq \mathrm{TuckRank}(\mathcal{A}_n)_n.
	\end{equation}
	
	Assume that $\mathrm{GTriRank}(\mathcal{A}_1) = \bar{r}$. Then there are  $\hat{\mathcal{B}}_1 \in \Re^{I_1 \times \bar{r} \times...\times \bar{r}}$, $\hat{\mathcal{B}}_2 \in \Re^{\bar{r} \times r \times...\times \bar{r}}$,..., and $\hat{\mathcal{B}}_N \in \Re^{\bar{r} \times \bar{r}\times... \times r}$ such that $\mathcal{A}_1 = [\hat{\mathcal{B}}_1 \hat{\mathcal{B}}_2... \hat{\mathcal{B}}_N]$. Replacing $\mathcal{X}$ and $\mathcal{A}_1$ in the first inequality of (\ref{E10}) by $\mathcal{A}_1$ and $\hat{\mathcal{B}}_1$, we have $\mathrm{TuckRank}(\mathcal{A}_1)_1 \leq \mathrm{TuckRank}(\hat{\mathcal{B}}_1)_1$. Note that $\hat{\mathcal{B}}_1 \in \Re^{n_1 \times \bar{r} \times...\times \bar{r}}$. By the definition of the Tucker rank, $\mathrm{TuckRank}(\hat{\mathcal{B}}_1)_1$ is the rank of an $n_1 \times \bar{r}^{N-1}$ matrix. Hence, $\mathrm{TuckRank}(\hat{\mathcal{B}}_1)_1 \leq \bar{r}^{N-1}$. In a similar way, we can prove 
	\begin{equation}
		\mathrm{TuckRank}(\mathcal{A}_n)_n \leq \mathrm{GTriRank}(\mathcal{A}_n)^{N-1}.
	\end{equation}
	
	Since $\mathcal{A}_1 = [\hat{\mathcal{B}}_1 \hat{\mathcal{B}}_2... \hat{\mathcal{B}}_N]$ and $\mathcal{A}_1 \in \Re^{n_1 \times r \times...\times r}$, it follows from  Theorem 3.1 that $\mathrm{GTriRank}(\mathcal{A}_1) \leq r = \mathrm{GTriRank}(\mathcal{X})$. Then for $n=2,3,...,N$, we can prove in a similar way. Therefore, the third inequality of Theorem 3.6 holds.
\end{proof}

	Next, we introduce the concept of hybrid multiple decomposition.
	\begin{defn}\textbf{(Hybrid multiple decomposition)}
		Suppose  the tensor $\mathcal{X}\in \Re^{I_1\times I_2\times...\times I_N}$ has Multiple rank
		 $\mathrm{MtpRank}(\mathcal{X}) = (R_1,R_2,...,R_N)$ and Tucker rank $\mathrm{TuckRank}(\mathcal{X})=(r_1,r_2,...,r_N)$. Consider a Tucker decomposition of $\mathcal{X}$:
		$
			\mathcal{X} = \mathcal{D} \times_1 \textbf{U}_1 \times_2 \textbf{U}_2 \times... \times_N \textbf{U}_N,
		$
		where $\mathcal{D} \in \Re^{r_1 \times r_2 \times...\times r_N}$ and  $\textbf{U}_1 \in \Re^{I_1 \times r_1}$, $\textbf{U}_2 \in \Re^{I_2 \times r_2}$,..., $\textbf{U}_N \in \Re^{I_N \times r_N}$. Furthermore, if $ \textbf{U}_i^\top \textbf{U}_i $, $ i = 1,2,...,N $ are invertible and the core tensor $\mathcal{D}$ has Multiple decomposition $\mathcal{D}=[\mathcal{A}_1\mathcal{A}_2...\mathcal{A}_N]$ with $\mathcal{A}_1 \in \Re^{r_1 \times R_2\times ... \times R_N}$,
		$\mathcal{A}_2 \in \Re^{R_1 \times r_2\times... \times R_N}$,..., and $\mathcal{A}_N \in \Re^{R_1 \times R_2 \times...\times r_N}$, then we call
		$$
		\mathcal{X} = [\mathcal{A}_1\mathcal{A}_2...\mathcal{A}_N] \times_1 \textbf{U}_1 \times_2 \textbf{U}_2 \times... \times_N \textbf{U}_N  
		= [(\mathcal{A}_1\times_1 \textbf{U}_1)(\mathcal{A}_2\times_2 \textbf{U}_2)...(\mathcal{A}_N\times_N \textbf{U}_N)]
		$$
		a \textit{hybrid multiple decomposition} of $\mathcal{X}$. 
	\end{defn}
	
	The hybrid multiple decomposition can further exploit the latent correlations along the long dimensions of the factor tensors in Multiple decomposition. This  is especially effective when strong correlations exist within individual modes.
	
	\subsubsection{The relation between Multiple  decomposition and  tensor network decompositions}
	Next, we compare Multiple decomposition with several classical tensor network decompositions, including Tensor Train (TT) decomposition \cite{TT}, Tensor Ring (TR) decomposition \cite{TR} and Fully Connected Tensor Network (FCTN) decomposition \cite{FCTN}. Figure \ref{CF2} illustrates the schematic diagrams of these decompositions.
	From the figure, we can find that
	\begin{itemize}
		\item TT and TR decompositions rely on sequential, nearest-neighbor interactions, which may hinder direct modeling of long-range or cross-mode correlations. Although Multiple decomposition does not establish explicit connections between every pair of factor tensors, it enables global interactions through shared virtual indices in the summation process.
		
		\item The FCTN decomposition establishes interactions between every pair of factor tensors, enabling it to fully capture the correlations among all modes. However, when decomposing an $ N $-th order tensor, the FCTN rank involves a vector with $ N(N-1)/2 $ elements, leading to significantly increased computational complexity.
		In contrast,  Multiple decomposition captures global correlations  via implicit coupling through shared latent dimensions, while the Multiple rank takes the form of an $N$-dimensional vector, resulting in lower storage and computational costs for high-order tensors.
	\end{itemize}
	\begin{figure}[t]
		\centering
		\includegraphics[width=6.3cm]{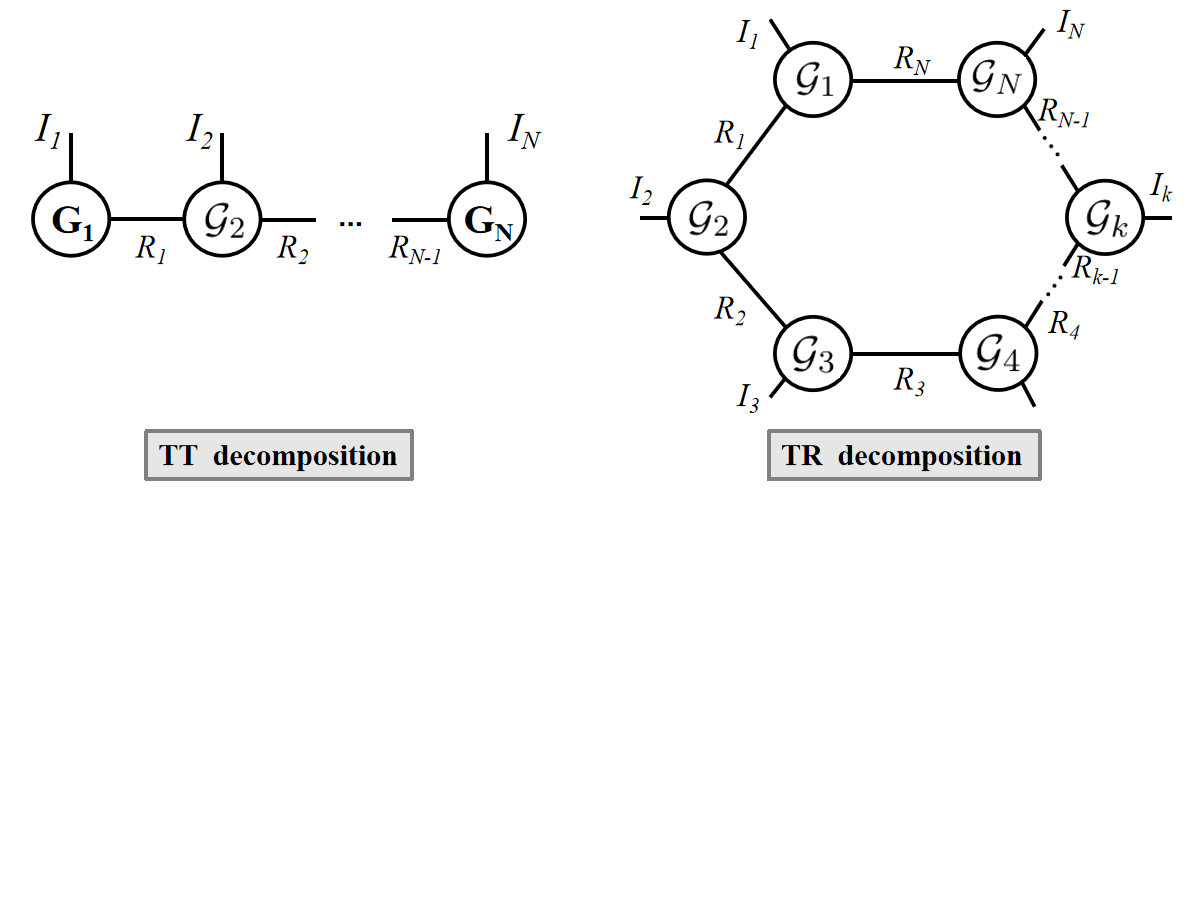}
		\includegraphics[width=6.3cm]{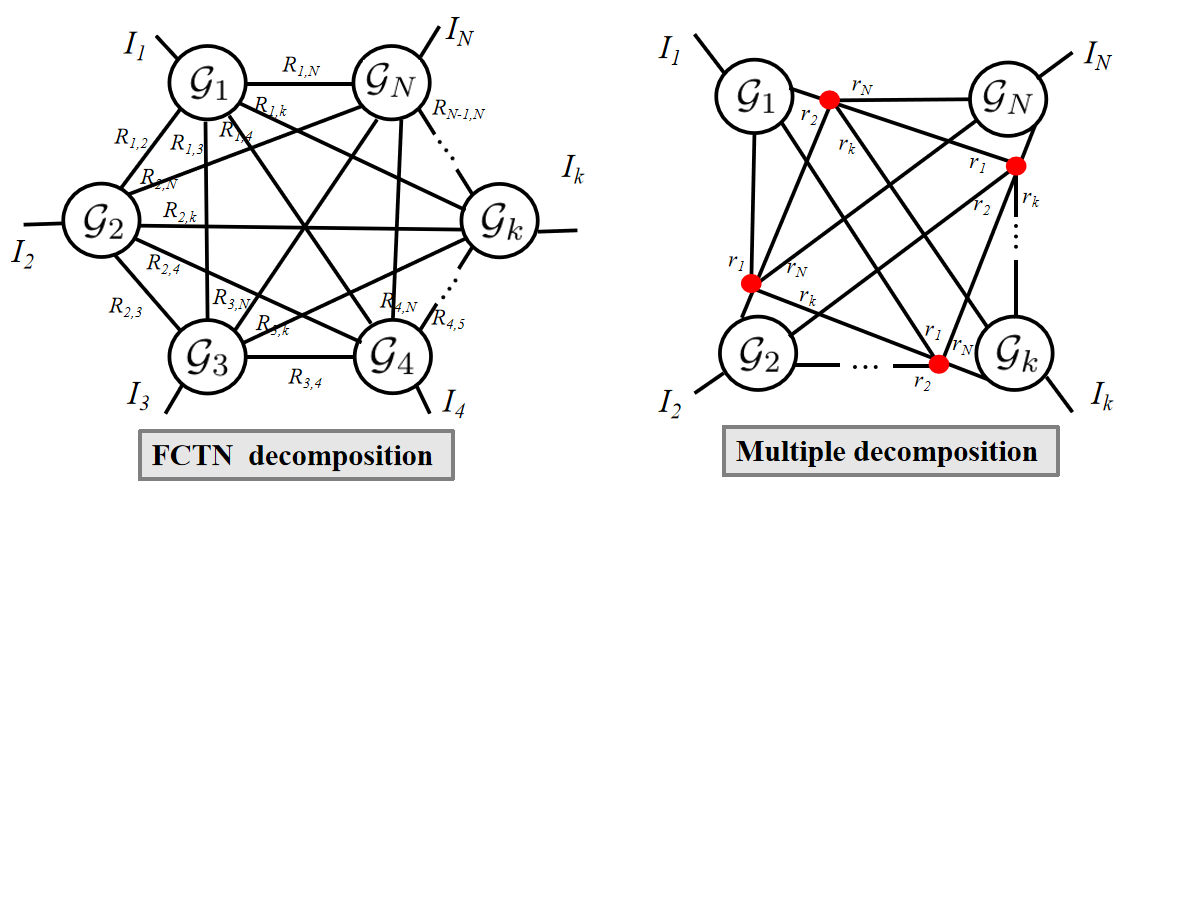}
		\vspace{-60pt}
		\caption{ \small Illustration of the TT, TR, FCTN and the proposed Multiple decomposition.   }
		\label{CF2}
		\vspace{-15pt}
	\end{figure}
	For third-order tensors, the TR, FCTN, and  Multiple decompositions  share  structural similarities. Triple decomposition imposes the constraint that the short-side dimensions of the factor tensors must be identical, which can be viewed as a constrained instance. In the case of tensors with order four or above, the TR, FCTN, and Multiple decompositions exhibit distinct structural and representational characteristics. Due to the distinct modeling mechanism, Multiple decomposition is not a special case of TR or FCTN, nor is it fully encompassed by them in higher-order settings.
	
	\subsection{Implicit Multiple  Tensor Decomposition}
	Non-grid data, such as point clouds, may also possess intrinsic low-rank structures \cite{Zhu2022}. Current tensor decompositions are primarily designed for regular grid-structured data and cannot be directly applied to non-grid data.   To address this issue, we integrate Multiple decomposition with implicit neural representation (INR), which  models data as a continuous function of coordinates. This leads to the proposed Implicit Multiple Tensor Decomposition (IMTD). Specifically, we regard the data $\mathcal{X}$ as a bounded tensor function $$f_{\mathcal{X}}(\cdot):D_\mathcal{X}:=\Omega_1\times \Omega_2\times...\times \Omega_N\to \Re.$$ 
	When $\mathcal{X} \in \Re^{I_1 \times I_2 \times \cdots \times I_N}$ is a regular grid tensor, each index set $\Omega_n$ can be defined as $\Omega_n := \{1, 2, \ldots, I_n\}$ for $n = 1, 2, \ldots, N$. In the case of non-grid data,  $\Omega_n$ are continuous subsets of $\Re$, and observations are given at irregularly sampled coordinates within $D_\mathcal{X}$. Then, we represent each factor tensor $ \mathcal{A}_n (n = 1,2,...,N) $ as a continuous function $\mathcal{A}_{\Theta_n}(\cdot)$ via a neural network parameterized by $\Theta_n$.  Once the parameters $\Theta_1, \Theta_2,...,\Theta_N$ are trained, we can obtain the factor tensor for the desired dimension by inputting the coordinates of that dimension. 
	\begin{figure}[t]
		\vspace{-5pt}
		\centering
		\!\!\!\!\!\!\!\includegraphics[width=8.3cm]{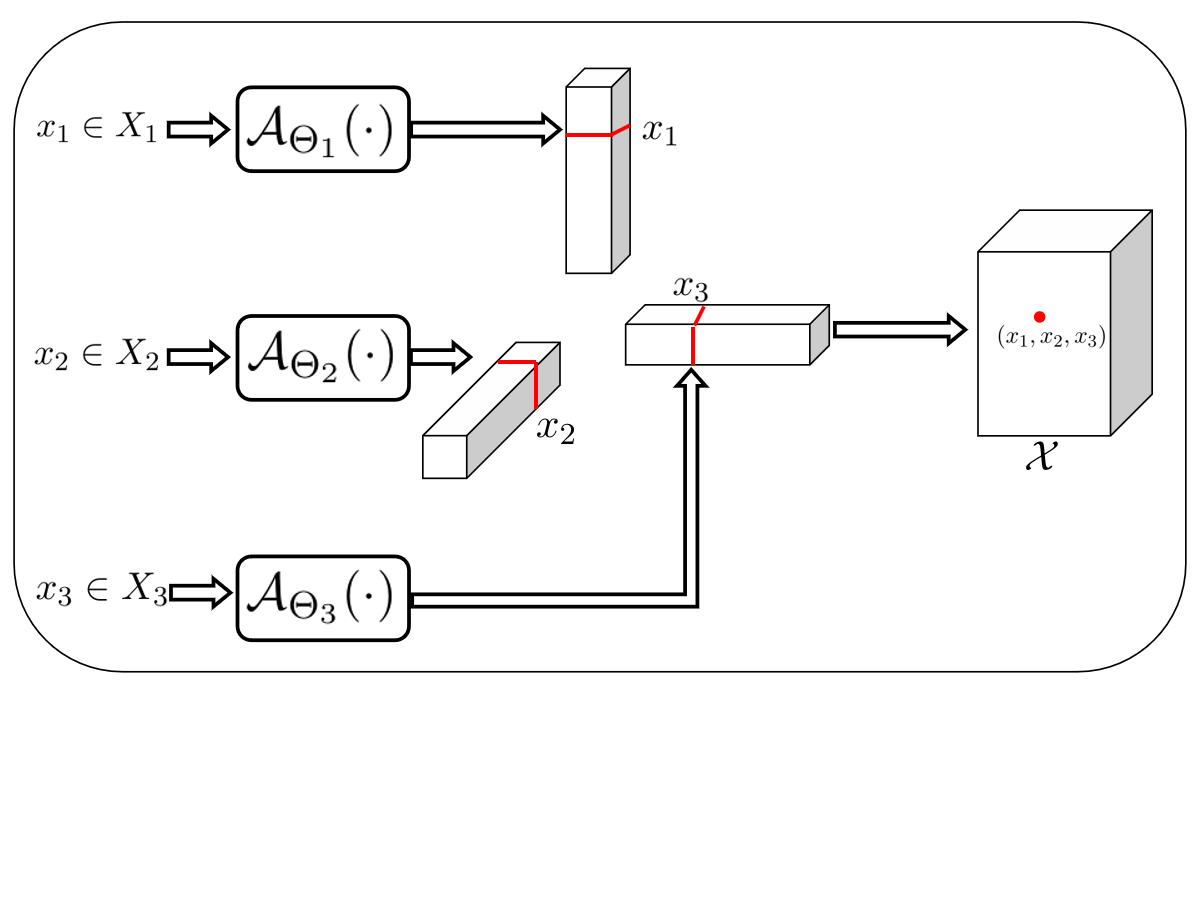}
		
		\vspace{-45pt}
		\caption{ \small The flowchart of the IMTD (taking a third-order tensor as an example).   }
		\label{CF3}
		\vspace{-20pt}
	\end{figure}
	Taking $ \mathcal{A}_{\Theta_1}(\cdot) $ as an example, it can be formulated as
	\begin{equation}\label{D13}
		\mathcal{A}_{\Theta_1}(x_1) := \text{reshape}(\textbf{H}_d (\sigma(\textbf{H}_{d-1} \cdots \sigma(\textbf{H}_1 x_1)) )) : \Omega_1 \to \Re^{1\times r_2\times r_3\times...\times r_N},
	\end{equation}
	where $ \Omega_1 \subset \Re $ is the definition domain in the first dimension,  $ \sigma(\cdot) $ is the nonlinear activation function, the `reshape' operation reorganizes the vector into a tensor of the specified size, and $ \Theta_1 := \{\textbf{H}_i\}_{i=1}^{d} $ are learnable weight matrices of the Multilayer Perceptron (MLP).  The MLP-parameterized IMTD is formulated as
	\begin{equation}\label{D14}
		f_{\mathcal{X}}(\textbf{x})=[\mathcal{A}_{\Theta_1}(x_1)\mathcal{A}_{\Theta_2}(x_2)...\mathcal{A}_{\Theta_N}(x_N)]\in \Re,
	\end{equation}
	where $\textbf{x}=(x_1,x_2,...,x_N)\in \Re^n$ is the coordinate of the observation point.  The flowchart of the IMTD for the third-order tensor case is shown in Figure \ref{CF3}. Through the above procedure, we can solve the IMTD by optimizing the network parameters.
	
	\begin{remark}
		By leveraging Property 3 of Theorem \ref{T22}, the factor tensors can be decoupled and  optimized within  Block Coordinate Descent (BCD) framework, with convergence guarantees established following the analysis in~\cite{TriB}. However, this approach relies on  grid-structured factor tensors. In contrast, IMTD represents the factor tensors as continuous implicit functions, enabling the model to generalize across arbitrary coordinates. Therefore, IMTD can also capture complex nonlinear patterns and uncover the underlying low-rank structure in non-grid data.
	\end{remark}

	\begin{theorem}\label{T6}
		Let the mappings $\mathcal{A}_{\Theta_1}(\cdot): \Omega_1 \rightarrow \mathbb{R}^{1\times r_2\times r_3\times...\times r_N}$, $\mathcal{A}_{\Theta_2}(\cdot): \Omega_2 \rightarrow \mathbb{R}^{r_1\times 1\times r_3\times...\times r_N}$,..., and $\mathcal{A}_{\Theta_N}(\cdot): \Omega_N \rightarrow \mathbb{R}^{r_1\times r_2\times...\times r_{N-1}\times 1}$ be $N$ MLPs structured as in (\ref{D13}) with parameters $\Theta_1$, $\Theta_2$,..., $\Theta_N$, where $\Omega_1, \Omega_2,..., \Omega_N \subset \mathbb{R}$ are bounded. Let the MLPs share the same activation function $\sigma(\cdot)$ and depth $d$. We assume that
		\begin{itemize}
			\item $\sigma(\cdot)$ is Lipschitz continuous  with Lipschitz constant $\kappa$.
			\item The $\ell_1$ norm of each weight matrix $\mathbf{H}_i$ in the  MLPs is bounded by $\omega > 0$.
		\end{itemize}
		Define the function $f(\cdot): D_f = \Omega_1 \times \Omega_2 \times...\times \Omega_N \rightarrow \mathbb{R}$ as $f(\cdot) := [\mathcal{A}_{\Theta_1}\mathcal{A}_{\Theta_2}...\mathcal{A}_{\Theta_N}](\cdot)$. Then, 
		for any $\textbf{x}=(x_1, x_2,..., x_N), \textbf{y}=(y_1, y_2,..., y_N) \in D_f$, one has
		$$
		|f(\textbf{x}) - f(\textbf{y})|\leq \delta\Vert\textbf{x}-\textbf{y} \Vert,
		$$
		i.e., $f$ is Lipschitz continious in $D_f$, where $\delta = \sqrt{2}\omega^{Nd} \kappa^{Nd-N} \zeta^{N-1}$ and $\zeta = \Vert \textbf{x} \Vert_{\infty}$. 
	\end{theorem}
    \begin{proof}
	For any $(x_1, x_2,..., x_N), (y_1, y_2,..., y_N) \in D_f$, direct calculation yields
	\begin{equation*}
		\begin{aligned}
			&~\quad|f(x_1, x_2,..., x_N) - f(y_1, x_2,..., x_N)| 
			\\&= |[\mathcal{A}_{\Theta_1}(x_1)\mathcal{A}_{\Theta_2}(x_2)...\mathcal{A}_{\Theta_N}(x_N)]-[\mathcal{A}_{\Theta_1}(y_1)\mathcal{A}_{\Theta_2}(x_2)...\mathcal{A}_{\Theta_N}(x_N)]|\\
			&=| [(\mathcal{A}_{\Theta_1}(x_1)-\mathcal{A}_{\Theta_1}(y_1))\mathcal{A}_{\Theta_2}(x_2)...\mathcal{A}_{\Theta_N}(x_N)]|\\
			&\le \Vert \mathcal{A}_{\Theta_1}(x_1)-\mathcal{A}_{\Theta_1}(y_1) \Vert_{\ell_1}\Vert \mathcal{A}_{\Theta_2}(x_2) \Vert_{\ell_1}...\Vert \mathcal{A}_{\Theta_N}(x_N) \Vert_{\ell_1}.
		\end{aligned}
	\end{equation*}
	Note that $\sigma(\cdot)$ is Lipschitz continuous, i.e., $|\sigma(x) - \sigma(y)| \leq \kappa |x - y|$ holds for any $x, y$, and letting $y = 0$ derives $|\sigma(x)| \leq \kappa |x|$ since $\sigma(0) = \sin(0) = 0$. Thus, we have
	\begin{equation*}
		\begin{aligned}
			\|\mathcal{A}_{\Theta_2}(x_2)\|_{\ell_1} &= \| \mathbf{H}_d^{(2)} (\sigma(\mathbf{H}_{d-1}^{(2)} \cdots \mathbf{H}_2^{(2)} (\sigma(\mathbf{H}_1^{(2)} x_2))))\|_{\ell_1}\\
			&\leq \omega \|\sigma(\mathbf{H}_{d-1}^{(2)} \cdots \mathbf{H}_2^{(2)} (\sigma(\mathbf{H}_1^{(2)} x_2)))\|_{\ell_1}\\
			& \leq \omega \kappa \|\mathbf{H}_{d-1}^{(2)} \cdots \mathbf{H}_2^{(2)} (\sigma(\mathbf{H}_1^{(2)} x_2))\|_{\ell_1}\\
			&...\\
			& \leq \omega^d \kappa^{d-1} |x_2|,
		\end{aligned}
	\end{equation*}
	where $\{\mathbf{H}_i^{(2)}\}_{i=1}^d$ are the weight matrices of $\mathcal{A}_{\Theta_2}(\cdot)$. Similarly, we have $\|\mathcal{A}_{\Theta_n}(x_n)\|_{\ell_1} \leq \omega^d \kappa^{d-1} |x_n|$ for $n = 3,4,...,N$. Meanwhile, it holds that
	\begin{equation*}
		\begin{aligned}
			&\hspace{-1cm}\quad~\|\mathcal{A}_{\Theta_1}(x_1) - \mathcal{A}_{\Theta_1}(y_1)\|_{\ell_1} \\
			&\hspace{-1cm}= \|\mathbf{H}_d^{(1)} (\sigma(\mathbf{H}_{d-1}^{(1)} \cdots \mathbf{H}_2^{(1)} (\sigma(\mathbf{H}_1^{(1)} x_1)))) - \mathbf{H}_d^{(1)} (\sigma(\mathbf{H}_{d-1}^{(1)} \cdots \mathbf{H}_2^{(1)} (\sigma(\mathbf{H}_1^{(1)} y_1))))\|_{\ell_1}\\
			&\hspace{-1cm}= \|\mathbf{H}_d^{(1)} (\sigma(\mathbf{H}_{d-1}^{(1)} \cdots \mathbf{H}_2^{(1)} (\sigma(\mathbf{H}_1^{(1)} x_1))) - \sigma(\mathbf{H}_{d-1}^{(1)} \cdots \mathbf{H}_2^{(1)} (\sigma(\mathbf{H}_1^{(1)} y_1))))\|_{\ell_1}\\
			&\hspace{-1cm}\leq \omega \|\sigma(\mathbf{H}_{d-1}^{(1)} \cdots \mathbf{H}_2^{(1)} (\sigma(\mathbf{H}_1^{(1)} x_1))) - \sigma(\mathbf{H}_{d-1}^{(1)} \cdots \mathbf{H}_2^{(1)} (\sigma(\mathbf{H}_1^{(1)} y_1)))\|_{\ell_1}\\
			&\hspace{-1cm}\leq \omega \kappa \|\mathbf{H}_{d-1}^{(1)} \cdots \mathbf{H}_2^{(1)} (\sigma(\mathbf{H}_1^{(1)} x_1)) - \mathbf{H}_{d-1}^{(1)} \cdots \mathbf{H}_2^{(1)} (\sigma(\mathbf{H}_1^{(1)} y_1))\|_{\ell_1}\\
			&\hspace{-0cm}...\\
			&\hspace{-1cm}\leq \omega^d \kappa^{d-1} |x_1 - y_1|,
		\end{aligned}
	\end{equation*}
	where $\{\mathbf{H}_i^{(1)}\}_{i=1}^d$ denote the weight matrices of $\mathcal{A}_{\Theta_1}(\cdot)$. Combining the above inequalities, we have
	\begin{equation*}
		\begin{aligned}
			|f(x_1, x_2,..., x_N) - f(y_1, x_2,..., x_N)|  &\leq \omega^{Md} \kappa^{Nd-N} |x_2|... |x_N| |x_1 - x_2|\\
			&\leq \omega^{Nd} \kappa^{Nd-N} \zeta^{N-1} |x_1 - y_1|.	
		\end{aligned}
	\end{equation*}
	In a similar way, we can show that
	\begin{equation*}
		|f(x_1, x_2,..., x_N) - f(x_1, y_2,..., x_N)| \leq \omega^{Nd} \kappa^{Nd-N} \zeta^{N-1} |x_2 - y_2|,
	\end{equation*}
	\begin{equation*}
		|f(x_1, x_2,..., x_N) - f(x_1, x_2,..., y_N)| \leq \omega^{Nd} \kappa^{Nd-N} \zeta^{N-1} |x_N - y_N|.
	\end{equation*}
	Based on the above inequalities, we can deduce
	$$
	|f(\textbf{x}) - f(\textbf{y})|\leq \delta\Vert\textbf{x}-\textbf{y} \Vert,
	$$
	for any $\textbf{x}=(x_1, x_2,..., x_N), \textbf{y}=(y_1, y_2,..., y_N) \in D_f$, which completes the proof.
\end{proof}
	
	\section{Model, algorithm and convergence analysis}\label{E}
	In this section, we investigate the tensor recovery models within the IMTD framework, along with corresponding algorithms and convergence analysis.
	
	\subsection{The tensor reconstruction problems within IMTD framework}
	We first consider the case where the model involves only a single variable. Within the IMTD framework, the models can be generally formulated as
	\begin{equation}\label{D16}
		\min_{\Theta}~ F(\mathcal{A}_{\Theta}(\textbf{x});\mathcal{M})+\lambda\phi(\mathcal{A}_{\Theta}(\textbf{x})),
	\end{equation}
	where $\mathcal{M}$ is the observed data, $\mathcal{A}_{\Theta}(\cdot):=[\mathcal{A}_{\Theta_1}\mathcal{A}_{\Theta_2}...\mathcal{A}_{\Theta_N}](\cdot)$ is a tensor function representing the reconstructed tensor in IMTD form, $\Theta:=(\Theta_1, \Theta_2, ...,\Theta_N)$ denotes the network parameter, \textbf{x} is the coordinates of  tensor, $F(\cdot)$  represents the fidelity term and $\phi(\cdot)$ is the regularization term, $\lambda>0$ is a weight parameter.
	We also consider the  tensor reconstruction problem with two variates, which can  be formulated as:
	\begin{equation}\label{D166}
			\min_{\Theta, \mathcal{E}}~ G(\Theta, \mathcal{E}):= F(\mathcal{A}_{\Theta}(\textbf{x}), \mathcal{E};\mathcal{M})+\lambda\phi(\mathcal{A}_{\Theta}(\textbf{x}))+\gamma\psi(\mathcal{E}),
	\end{equation}
	where $\mathcal{E}$  is typically treated as a perturbation variable, $\psi(\mathcal{E})$ is a regularization term, and $\gamma>0$ is a weight parameter. We assume that $F$ is continuously differentiable,  $\phi$ is lower semi-continious (LSC), $\psi$ is LSC and coercive. Next, we present two classical reconstruction problems under the framework of  (\ref{D16}) and (\ref{D166}).
	
	\textit{1) Robust Tensor Completion (RTC).} This
	 is a classical problem on the original meshgrid tensor, aiming to recover the underlying low rank tensor from partial observations. Specifically, given an observed tensor $\mathcal{M}$ corrupted by sparse noise and an index set $\Omega$ of observed entries, the RTC problem can be formulated within the framework (\ref{D166}) as follows:
	\begin{equation}\label{D17}
		\begin{aligned}
			\min_{\Theta, \mathcal{E}}~ \Vert P_{\Omega}(\mathcal{A}_{\Theta}(\textbf{x})+\mathcal{E}-\mathcal{M}) \Vert_F^2+\lambda\Vert \mathcal{A}_{\Theta}(\textbf{x}) \Vert_{\rm TV\ell_1}+\gamma\Vert\mathcal{E}\Vert_{\ell_1},
		\end{aligned}
	\end{equation}
	where  $\Vert \cdot \Vert_{\rm TV\ell_1}:=\Vert \mathcal{D}(\cdot) \Vert_{\ell_1}$ denotes the Total Variation (TV) regularization in the  $\ell_1$ sense,  $\mathcal{D}(\cdot)$ represents a discrete difference operator that computes the first-order  differences along each mode of the tensor \cite{Bai2016}.

	\textit{2) Point Cloud Upsampling} (PCU)
	refers to the process of converting a sparse point cloud into a dense one, which can benefit subsequent applications \cite{Luo2024}. 
	For a sparse point cloud $\mathbf{P} \in \mathbb{R}^{p \times 3}$, where $p$ denotes the number of points, we use $\Omega := \{\mathbf{P}_{(m,:)}\}_{m=1}^p$ to denote the observed set, and adopt  signed distance function (SDF) \cite{SDF} to learn a continuous representation from the sparse point cloud. The PCU model within the IMTD framework can be expressed as
	\begin{equation*}
		\begin{aligned}
			&\min_{s} ~ \sum_{\mathbf{x} \in \Omega} |s(\mathbf{x})| + \lambda \int_{\mathbb{R}^3} \left|\left\|\frac{\partial s(\mathbf{x})}{\partial \mathbf{x}}\right\|_F^2 - 1\right| d\mathbf{x} 
			+ \gamma \int_{\mathbb{R}^3 \setminus \Omega} \exp(-|s(\mathbf{x})|) d\mathbf{v},\\
			&~s.t. ~~ s(\textbf{x})=\mathcal{A}_{\Theta}(\textbf{x}):=[\mathcal{A}_{\Theta_1}\mathcal{A}_{\Theta_2}...\mathcal{A}_{\Theta_N}](\textbf{x}),
		\end{aligned}
	\end{equation*}
	where $s(\cdot): \mathbb{R}^3 \to \mathbb{R}$ denotes the SDF,  $\lambda$ and $\gamma$ are weight parameters. The first term enforces the SDF to be zero on the observed points, which represents the underlying surface of the point cloud. The second term encourages the magnitude of the SDF gradients to be unity everywhere, promoting well-behaved geometry. The third term penalizes deviations of the SDF values from zero outside the observed point set \cite{INR, Luo2024}.  To generate a dense point cloud, we sample points $\mathbf{v}$ uniformly across the space such that $|s(\mathbf{v})| < \tau$, where $\tau$ is a predefined threshold. These sampled points constitute the desired upsampling result.
	
	\subsection{Algorithm and convergence analysis}
	Since $\Theta$ is the true optimization variable in $\mathcal{A}_{\Theta}(\textbf{x})$, we  denote $\mathcal{A}_{\Theta}(\textbf{x})$ as $\mathcal{A}_\textbf{x}(\Theta)$ or $\mathcal{A}_{\Theta}$.
	To solve the  model (\ref{D16}), we can  define the loss function as
	$
	 F(\mathcal{A}_{\Theta};\mathcal{M})+\lambda\phi(\mathcal{A}_{\Theta}),
	$
	and directly optimize  $\Theta$ using adaptive moment estimation (Adam) \cite{ADAM} algorithm. However, for the  model (\ref{D166}), directly optimizing a multi-variable loss function may cause conflicts, leading to degraded performance.
	Therefore,  we employ the  Proximal Alternating Least Squares (PALS) algorithm to solve the model.  The  iterative scheme is
	\begin{align}
			\vspace{-3pt}
		\Theta^{k+1} &\in \mathrm{arg}\min_{\Theta}~ F(\mathcal{A}_{\Theta}, \mathcal{E}^k;\mathcal{M})+\lambda\phi(\mathcal{A}_{\Theta})+\frac{\eta}{2}\Vert\mathcal{A}_{\Theta}-\mathcal{A}_{\Theta^k}\Vert_F^2, \label{D18}\\
		\mathcal{E}^{k+1} &\in \mathrm{arg}\min_{\mathcal{E}}~ \gamma\psi(\mathcal{E})+F(\mathcal{A}_{\Theta^{k+1}}, \mathcal{E};\mathcal{M})+\frac{\eta}{2}\Vert\mathcal{E}-\mathcal{E}^k\Vert_F^2, \label{D19}
		\vspace{-3pt}
	\end{align}
	where $\eta>0 $ is a  proximal parameter. It is worth noting that we add \( \frac{\eta}{2}\Vert\mathcal{A}_{\Theta}-\mathcal{A}_{\Theta^k}\Vert_F^2 \) rather than \( \frac{\eta}{2}\Vert\Theta-\Theta^k\Vert_F^2 \) to the $\Theta$ subproblem.
	The pseudocode of PAO algorithm for solving the  model (\ref{D166}) is summarized in Algorithm \ref{AL1}.
	\vspace{-3pt}
	\begin{algorithm}
		\caption{ PAO algorithm for solving the  model (\ref{D166})}
		{\bf Input:} ${\mathcal{M}}$,  $ \lambda, \gamma, \eta$.  
		
		\textbf{Initialize:} $\Theta^{0}$, ${\mathcal{E}}^{0}$. Set $k=0$.
		
		\textbf{Step1:} Compute $\Theta^{k+1}$  by  solving the subproblem (\ref{D18});
		
		\textbf{Step2:} Compute $\mathcal{E}^{k+1}$  by  solving the subproblem (\ref{D19});
		
		\textbf{Step3:} If a termination criterion is not met, set $k = k+1$  and return to step 1;
		
		{\bf Output:} $\Theta^{k+1}$, $\mathcal{A}_{\Theta^{k+1}}$.
		\label{AL1}
	\end{algorithm}
	\vspace{-2pt}
	
	Next, we provide a convergence analysis for Algorithm~\ref{AL1}. Since the objective function in MITD-based models may lack the Kurdyka–Łojasiewicz (KL) property\cite{KL3, Attouch} due to complex landscapes induced by deep architectures and nonlinear activations, we adopt a KL-free analysis framework to ensure broader applicability.   Define
	$$
	\widehat{G}(\mathcal{A},\mathcal{E}):=F(\mathcal{A}, \mathcal{E};\mathcal{M})+\lambda\phi(\mathcal{A})+\gamma\psi(\mathcal{E}).
	$$
	Then $G(\Theta, \mathcal{E})$ is the composite function of $\widehat{G}(\mathcal{A}, \mathcal{E})$ and $\mathcal{A}=\mathcal{A}_\textbf{x}(\Theta)$.

	\begin{theorem}\label{T1}
		Assume the hypotheses of Theorem \ref{T6} hold. Then, any limit point $(\Theta^{*}, \mathcal{E}^*)$ of the sequence $\{(\Theta^{k}, \mathcal{E}^k)\}$ generated by Algorithm \ref{AL1} is a critical point of $ G $.  If the map $\Theta \mapsto \mathcal{A}_\Theta$ is a submersion at $\Theta^*$, i.e., the Jacobian has full row rank, then $(\mathcal{A}_{\Theta^*}, \mathcal{E}^*)$ is also a critical point of $ \widehat{G} $, and  $\{(\mathcal{A}_{\Theta^{k}}, \mathcal{E}^k)\}$ cannot oscillate between distinct  critical points of $\widehat{G}$.  Moreover, if all critical points of $\widehat{G}$ are isolated, the sequence $\{(\mathcal{A}_{\Theta^{k}}, \mathcal{E}^k)\}$ converges globally to a single critical point of $\widehat{G}$.
	\end{theorem}
	\begin{proof}
		Since $\Theta^{k+1}$ and $\mathcal{E}^{k+1}$ are optimal solutions of (\ref{D18}) and (\ref{D19}), we have
		\begin{equation}\label{E1}
			\begin{aligned}
				G(\mathcal{A}_{\Theta^{k+1}}, \mathcal{E}^k) &\leq G(\mathcal{A}_{\Theta^{k}}, \mathcal{E}^k) - \frac{\eta}{2} \Vert\mathcal{A}_{\Theta^{k+1}} - \mathcal{A}_{\Theta^{k}}\Vert_F^2, \\
				G(\mathcal{A}_{\Theta^{k+1}}, \mathcal{E}^{k+1}) &\leq G(\mathcal{A}_{\Theta^{k+1}}, \mathcal{E}^k) - \frac{\eta}{2} \Vert\mathcal{E}^{k+1} - \mathcal{E}^k\Vert_F^2.
			\end{aligned}
		\end{equation}
		For convenience, we denote $G^k=G(\mathcal{A}_{\Theta^{k}}, \mathcal{E}^k)$. Then, from (\ref{E1}), one has
		\begin{equation}\label{D20}
			\begin{aligned}
				G^{k+1} \leq G^k - \frac{\eta}{2} (\Vert\mathcal{A}_{\Theta^{k+1}} - \mathcal{A}_{\Theta^{k}}\Vert_F^2 + \Vert\mathcal{E}^{k+1} - \mathcal{E}^k\Vert_F^2).
			\end{aligned}
		\end{equation}
		Define the Lyapunov function:
		$
		V_k := G^k + \frac{\eta}{2} a_k + \frac{\eta}{2} e_k,
		$
		where \(a_k = \|\mathcal{A}_{\Theta^{k}} - \mathcal{A}_{\Theta^{k-1}}\|_F^2\), \(e_k = \|\mathcal{E}^k - \mathcal{E}^{k-1}\|_F^2\).
		Then, by combining with (\ref{D20}), we analyze the difference:
		\begin{equation}\label{D21}
			\begin{aligned}
				V_{k+1} - V_k &= G^{k+1} - G^k + \frac{\eta(a_{k+1} - a_k+e_{k+1} - e_k))}{2} \le-\frac{\eta(a_k+e_k)}{2}<0.
			\end{aligned}
		\end{equation}
		This implies that \( V_k \) is monotonically decreasing. Therefore, \( V_k \) converges to some finite limit \( V^* \) since it is bounded below. By summing (\ref{D21}) over $k$, we obtain
		$
		\sum_{k=1}^{\infty} (a_k+e_k) \leq \frac{2}{\eta}(V_1 - V^*) < \infty,
		$
		which implies $a_k+e_k \to 0$, i.e.,
		$
		\|\mathcal{A}_{\Theta^{k+1}} - \mathcal{A}_{\Theta^{k}}\|_F \to 0, ~ \|\mathcal{E}^{k+1} - \mathcal{E}^k\|_F \to 0.
		$
		From the definition of $V^k$, we have $G^k\to V^*$. According to the assumption of Theorem \ref{T6}, the network parameter $\Theta$ is bounded.  Thus, the sequence $\{(\Theta^{k}, \mathcal{E}^k)\}$ generated by Algorithm \ref{AL1} is also bounded since $G$  is coercive in $\mathcal{E}$.   Therefore, there exists a subsequence $\{(\Theta^{t}, \mathcal{E}^t)\}$ that converges to a limit point, denoted as $(\Theta^{*}, \mathcal{E}^*)$.  From (\ref{D18}) and (\ref{D19}), one has:
		\begin{equation}\label{E2}
			\begin{aligned}
				0 &\in [\nabla_{\mathcal{A}_{\Theta}} F(\mathcal{A}_{\Theta^{t+1}}, \mathcal{E}^t) + \lambda\partial \phi(\mathcal{A}_{\Theta^{t+1}}) + \eta \mathcal{R}^t]^\top \mathcal{J}_{t+1},\\
				0 &\in \nabla_{\mathcal{E}} F(\mathcal{A}_{\Theta^{t+1}}, \mathcal{E}^{t+1})+\gamma \partial \psi(\mathcal{E}^{t+1}) + \eta (\mathcal{E}^{t+1} - \mathcal{E}^t),
			\end{aligned}
		\end{equation}
		where  $\mathcal{J}_{t+1}:=\left. \frac{\partial \mathcal{A}_\Theta}{\partial \Theta} \right|_{\Theta = \Theta^{t+1}}$, $\mathcal{R}^t:=\mathcal{A}_{\Theta^{t+1}} - \mathcal{A}_{\Theta^{t}}$. The derivative computation is performed in the convention of vectorizing the tensor variables. Then we  obtain:
		\begin{equation*}
			\begin{aligned}
				[\nabla_{\mathcal{A}_{\Theta}} F(\mathcal{A}_{\Theta^{t+1}}, \mathcal{E}^{t+1})-\nabla_{\mathcal{A}_{\Theta}} F(\mathcal{A}_{\Theta^{t+1}}, \mathcal{E}^{t})+\eta\mathcal{R}^t]^\top \mathcal{J}_{t+1} &\in \partial G_{\Theta}(\mathcal{A}_{\Theta^{t+1}}, \mathcal{E}^{t+1}),\\
				\eta(\mathcal{E}^t-\mathcal{E}^{t+1}) &\in \partial G_{\mathcal{E}}(\mathcal{X}_f^{t+1}, \mathcal{E}^{t+1}).
			\end{aligned}
		\end{equation*}
		Taking $k\to +\infty$ on both sides, we obtain $
			0\in \partial G_{\Theta}(\Theta^{*}, \mathcal{E}^{*}),~ \text{and} ~ 0\in \partial G_{\mathcal{E}}(\Theta^{*}, \mathcal{E}^{*})
		$, which implies that $(\Theta^{*}, \mathcal{E}^*)$ is a critical point of  $ G $. Since  $\mathcal{A}_\Theta$ is a submersion at $\Theta^*$, i.e., $\mathcal{J}_{*}:=\left. \frac{\partial \mathcal{A}_\Theta}{\partial \Theta} \right|_{\Theta = \Theta^{*}}$ is full row rank, we can also obtain $0\in \partial \widehat{G}_{\mathcal{A}}(\mathcal{A}_{\Theta^{*}},\mathcal{E}^{*})$ and $0\in \partial \widehat{G}_{\mathcal{E}}(\mathcal{A}_{\Theta^{*}}, \mathcal{E}^{*})$. Therefore, $(\mathcal{A}_{\Theta^*}, \mathcal{E}^*)$ is also a critical point of $ \widehat{G} $.
		
		Next, we prove that the sequence $\{(\mathcal{A}_{\Theta^k}, \mathcal{E}^k)\}$ does not oscillate among limit points. By contradiction, if it did, then $ a_k + e_k $ would remain greater than some constant $ C $ indefinitely. From (\ref{D21}), we know that in each iteration, $ V_{k+1} $ decreases by at least $ C \eta / 2 $, which contradicts the fact that $ V_k $ is bounded below. If critical points are all isolated, the sequence eventually enters and stays in a neighborhood of one such point, i.e., it converges globally to that critical point.
	\end{proof}
	\begin{remark}
		The parameter sequence $\{\Theta^k\}$ may fail to converge due to the non-uniqueness of network parameters, i.e., Multiple $\Theta$ can yield the same $\mathcal{A}_\Theta$. However, our primary interest lies in the convergence of the  output $\{\mathcal{A}_{\Theta^k}\}$. Even if $\Theta^k$ oscillates among equivalent parameterizations, the corresponding outputs remain stable, which is sufficient for practical performance and structural recovery.
	\end{remark}
	
	\section{Experiments}\label{NN}
	In this section, we conducted extensive numerical experiments to verify the effectiveness of the proposed IMTD, including  robust tensor completion (grid-structured data) and point cloud upsampling (non-grid data).  The  IMTD is trained on a NVIDIA GeForce RTX 4060 GPU, and we run the compared codes  in MATLAB R2023a on a 12th Gen Intel® Core™ i5-12450H processor at 2.00 GHz.
	
	\subsection{Parameter settings} 
	The  crucial parameters are the Multiple rank in  IMTD. Although the Multiple rank of a tensor $\mathcal{X} \in \Re^{I_1 \times I_2 \times \cdots \times I_N}$ cannot be determined directly, a reasonable range can be estimated. According to Theorem \ref{T11},  we recommend first roughly determining  GTriRank($\mathcal{X}$), and then gradually adjust it downward to select the  Multiple rank. Specifically, we propose the following lower and upper bounds for estimating  GTriRank($\mathcal{X}$):
	$$
	r_{min}:=\max\{I_1^{\frac{1}{N-1}}, I_2^{\frac{1}{N-1}},..., I_N^{\frac{1}{N-1}}\}/3,\quad r_{max}:=\frac{2}{3}\sqrt[N-1]{\frac{I_1I_2...I_N}{I_1+I_2+...+I_N}}.
	$$

	For a point cloud dataset $\mathcal{X}$ consisting of $N$ points, i.e., $\mathcal{X}\in \Re^{N\times 3}$, we recommend  $r = N^{\frac{1}{3}}$
	as an upper bound for the choice of MtpRank($\mathcal{X}$).

	\subsection{Comparison  on robust tensor completion (RTC) task}
	In this subsection, we compare IMTD with several  RTC methods, including: TriD \cite{TriB}, FCTN \cite{FCTN}, NTRC \cite{NTRC}, BCNRTC \cite{Zhao2} and RTCDLN \cite{RTCDLN}. All these methods are evaluated on both color images and color videos. Specifically, the experiments are conducted on two color images: Pepper ($512 \times 512 \times 3$) and Flower ($321 \times 481 \times 3$), along with two color videos: Ocean ($112 \times 160 \times 3 \times 32$) and Man ($144 \times 176 \times 3 \times 51$). All images and videos are corrupted by (i) randomly removing a fraction of entries (i.e., missing data) and (ii) adding salt-and-pepper noise to the observed entries. We set the sampling rates ($sr$) to $ 0.2, 0.4, $ and $ 0.6 $, and consider noise levels of $ \sigma \in \{0.2, 0.4, 0.6\} $.
	
	\subsubsection{Color images}
	Table \ref{T2} presents the reconstruction performance of various methods   for the color images.
	\begin{table}[h]\small
		\centering
		\renewcommand\arraystretch{1}
		\vspace{-10pt}
		\caption{\small PSNR (dB) values for restoring results of different methods for colour images corrupted by sample loss and salt-and-pepper noise.}
		\label{T2}
		\vspace{-5pt}
		\begin{tabular}{cccccccc}
			\hline
			Sample & Noise  & \multicolumn{6}{c}{Pepper}  \\ \cline{3-8} 
			rates (\%)& level (\%)  & TriD & FCTN & NTRC & BCNRTC  & RTCDLN& IMTD \\ \hline
			\multirow{3}{*}{20} & 20 & 18.86 & 19.62 & 19.76 & 21.17  & \underline{22.86} &  \textbf{28.65} \\
			& 40 & 16.03 & 16.36 & 16.95 & 17.49  & \underline{18.83} & \textbf{25.13} \\
			& 60 & 12.62 & 12.89 & 13.24 &  13.92 & \underline{14.26}& \textbf{21.89} \\ \cline{3-8}
			\multirow{3}{*}{40} & 20 &23.58 &24.27 &23.98 & 25.96 &\underline{26.58} &\textbf{31.27}  \\
			& 40 & 19.24 &20.16 & 19.70 &21.74 &\underline{22.17} & \textbf{28.08}\\
			& 60 & 15.63&15.87 &16.24 & 16.09 &\underline{16.81} & \textbf{24.17}\\\cline{3-8}
			\multirow{3}{*}{60}& 20 &27.75 &28.83 & 29.36 & \underline{30.92} & 29.85 &\textbf{32.86} \\
			& 40 &22.47 &23.72 &23.90 &\underline{26.16} & 25.97& \textbf{29.62} \\
			& 60 & 18.05 & 18.79&19.14 & \underline{20.79} & 18.85 & \textbf{24.79} \\\hline
			Sample & Noise  & \multicolumn{6}{c}{Flower}  \\ \cline{3-8} 
			rates (\%)& level (\%)  & TriD & FCTN & NTRC & BCNRTC  & RTCDLN & IMTD \\ \hline
			\multirow{3}{*}{20} & 20 & 19.15 & 20.71 & 21.35 & 22.07  & \underline{22.60} &  \textbf{26.75} \\
			& 40 &  15.20 & 15.45 & 16.71&  17.23 & \underline{19.82} & \textbf{24.67} \\
			& 60 & 10.63 & 11.37 & 11.42 & 12.48  & \underline{14.76} & \textbf{20.98} \\ \cline{3-8}
			\multirow{3}{*}{40} & 20 &22.49 & 23.89& 24.64 & \underline{26.02} &25.54 & \textbf{28.91} \\
			& 40 &17.71 & 18.74&19.51 &21.27 &\underline{21.66} &  \textbf{25.48}\\
			& 60 &13.46 &14.10 &14.84 &15.35 &\underline{15.94} &\textbf{22.14} \\\cline{3-8}
			\multirow{3}{*}{60}&20 & 25.01 &26.95 &27.55 & \underline{29.07} &28.12  & \textbf{30.42}\\
			& 40 &21.90 &22.43 &23.42 & \underline{24.17} &23.20 & \textbf{27.38} \\
			& 60 &17.15 &17.80 &18.37 & \underline{19.31} & 17.45& \textbf{23.69} \\\hline
		\end{tabular}
		\vspace{-5pt}
	\end{table}
	As shown in the table, the IMTD-based RTC method consistently outperforms all competing methods across all test settings. Notably, its performance advantage becomes even more significant under conditions of low sampling rates and high noise levels. Taking the `Pepper' dataset as an example, when the sampling rate  is 0.2 and the noise level  is 0.6, the PSNR of the image reconstructed by IMTD is 7.6 dB higher than that achieved by the second-best method, RTCDLN. This is mainly attributed to the robustness of the Multiple decomposition structure and the strong representation capability of the implicit neural network for the data. Moreover, TriD, FCTN and NTRC achieve very similar performance on both datasets, primarily because these methods exhibit comparable expressiveness for third-order data. TriD further restricts the short sides of the factor tensors to be identical, reducing flexibility and resulting in performance  inferior to other methods. As the sampling rate increases and the noise level decreases, the observed data become more informative, allowing traditional low-rank RTC models to better recover the underlying structure. Consequently, their performance improves significantly. Both BCNRTC and RTCDLN are based on the tensor nuclear norm (TNN): the former enhances TNN through a non-convex correction, while the latter introduces a  transform domain.   Their superior performance demonstrates the effectiveness of TNN-based RTC methods. However, their  performance still lags behind that of IMTD, particularly under strong interference conditions. The average PSNR  under three noise levels are displayed in Figure \ref{CF4}, from which we can more intuitively perceive the advantages of the IMTD-based RTC method.
	\begin{figure}
		\centering
		\includegraphics[width=12.7cm]{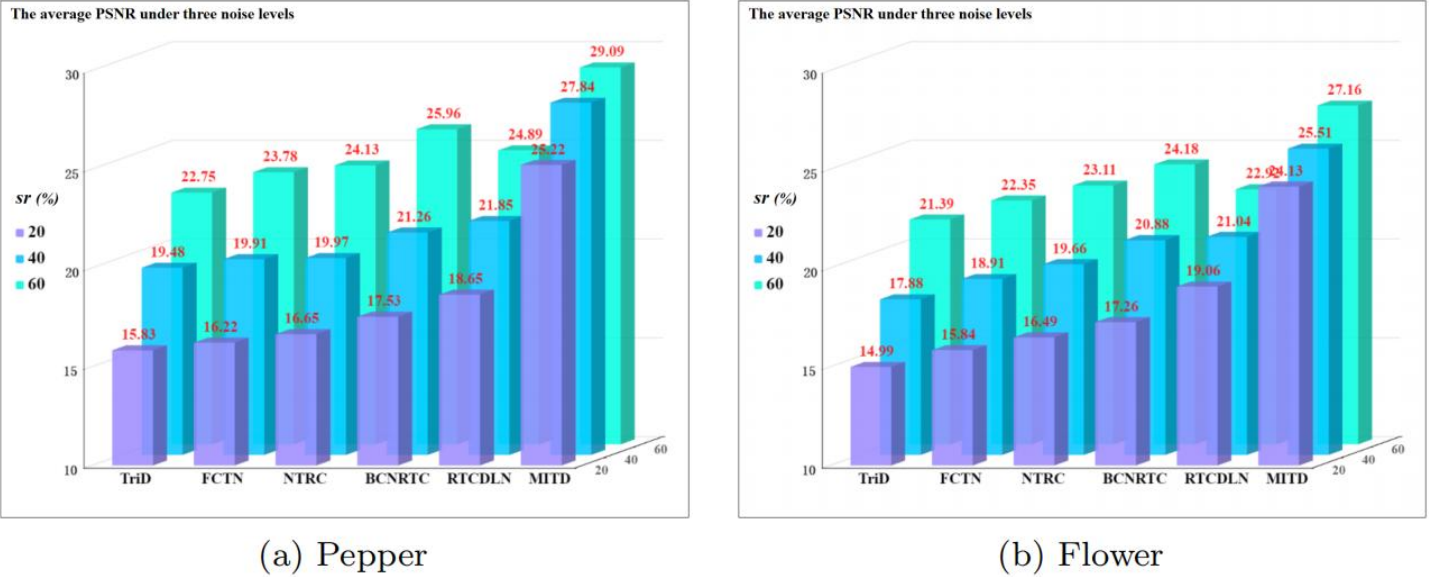}

		\vspace{-3pt}
		\caption{\small  The average PSNR  under three noise levels at different sampling rates.   }
		\label{CF4}
		\vspace{-15pt}
	\end{figure}
	Figure \ref{fig64} shows the recovered  images, where
	the region marked by  blue box is magnified and displayed in the upper left corner.
	
	\begin{figure*}[htbp]
		\vspace{-2pt}
		\centering
		
		\captionsetup[subfloat]{labelsep=none,format=plain,labelformat=empty}
		\includegraphics[width=12.5cm]{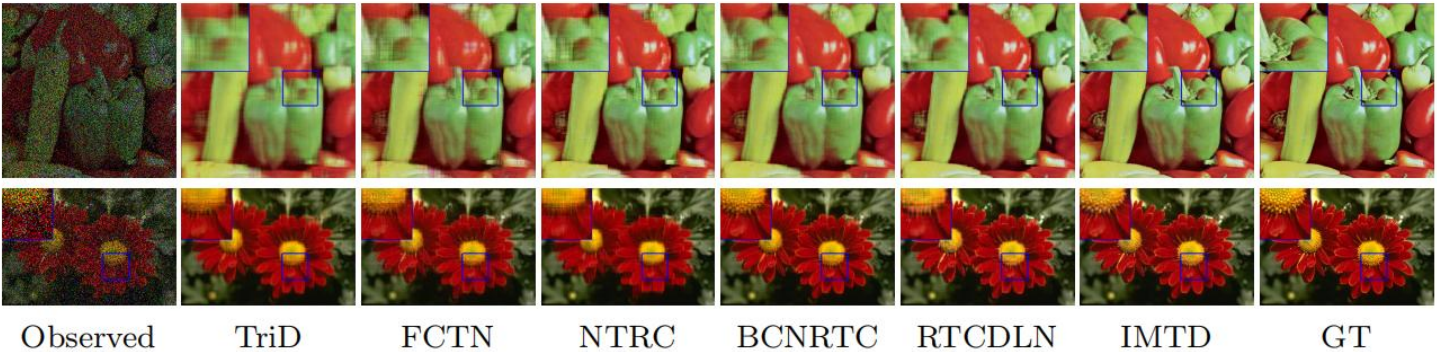}

		\vspace{-4pt}
		\caption{\small The visual comparison of restoration results on the `Pepper' (sr=0.4, $\sigma$=0.2) and `Flower' (sr=0.6, $\sigma$=0.2) datasets for each method. \tiny \label{fig64}}
	\end{figure*}
	
	\vspace{-15pt}
	As shown in the figure, the IMTD-based method recovers corrupted color images more effectively than the other compared approaches, producing visually clearer results. For example, in the blue-marked region of the `pepper' image that contains relatively complex textures, the low-rank RTC methods tend to oversmooth these intricate details when filling in missing values and suppressing noise. This results in distortions in the complex regions of the image. In contrast, the IMTD-based approach can better preserve and recover fine image details by implicitly learning the underlying functional distribution  of the data. For the `Flower' dataset, BCNRTC also demonstrates strong reconstruction performance in localized regions. This is primarily due to the non-convex surrogate of the TNN and $\ell_1$ norm, which allows for better preservation of the image's high-frequency details. However, achieving such performance typically requires careful and intricate parameter tuning.

	\subsubsection{Color videos}
	A color video can be represented as a fourth-order tensor of dimensions $W \times H \times 3 \times S$, where $W$ and $H$ denote the spatial width and height of each frame, $3$ corresponds to the three color channels (e.g., RGB), and $S$ represents the total number of frames in the video. For  methods that can only handle third-order tensors (TriD, BCNRTC, and RTCDLN), we reshape the  data into a third-order tensor with size W × H × (3S), where 3S combines the color channels across all frames. Figure \ref{CF5} presents the average PSNR values of the reconstructed images  across three noise levels in various sampling rates.

	\begin{figure}
		\centering
		\subfloat[Ocean]{\includegraphics[width=6cm]{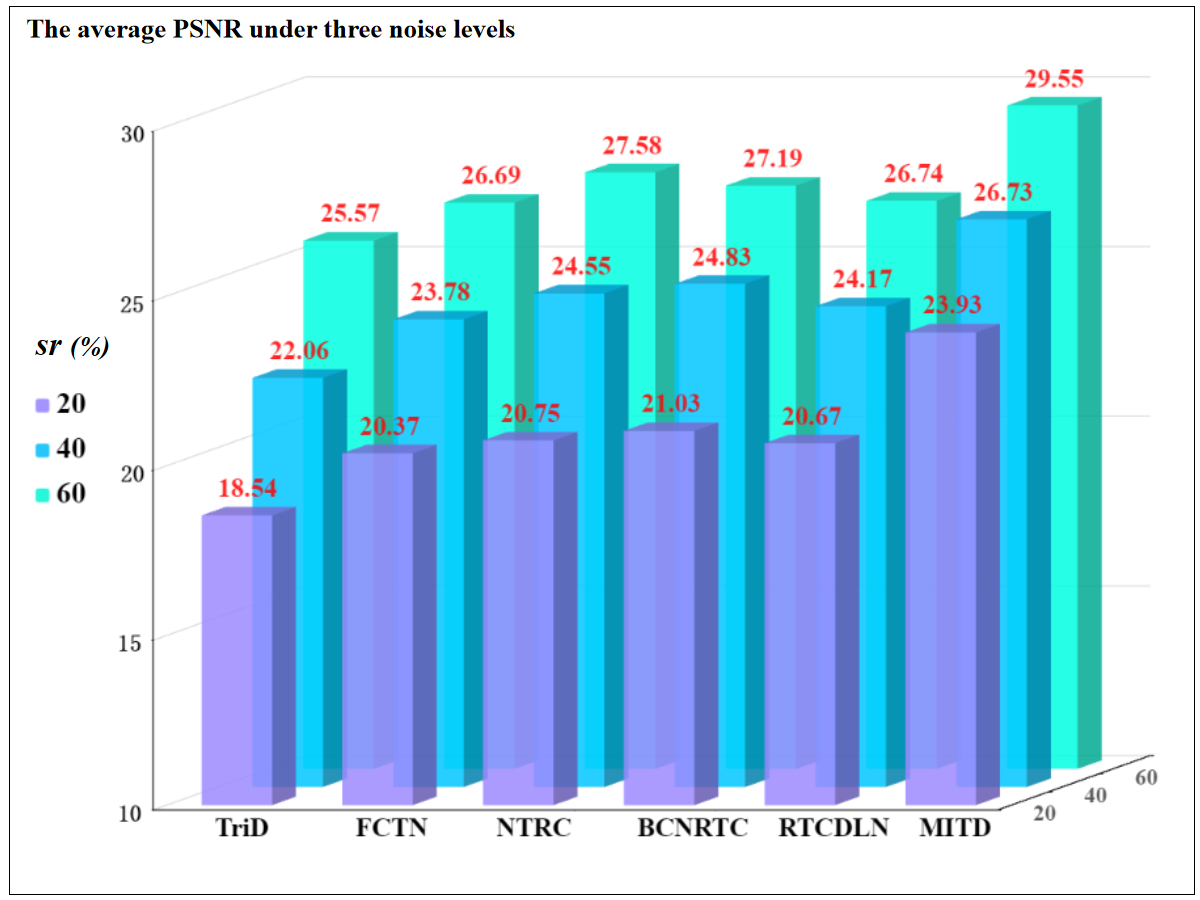}}\quad
		\subfloat[Man]{\includegraphics[width=6cm]{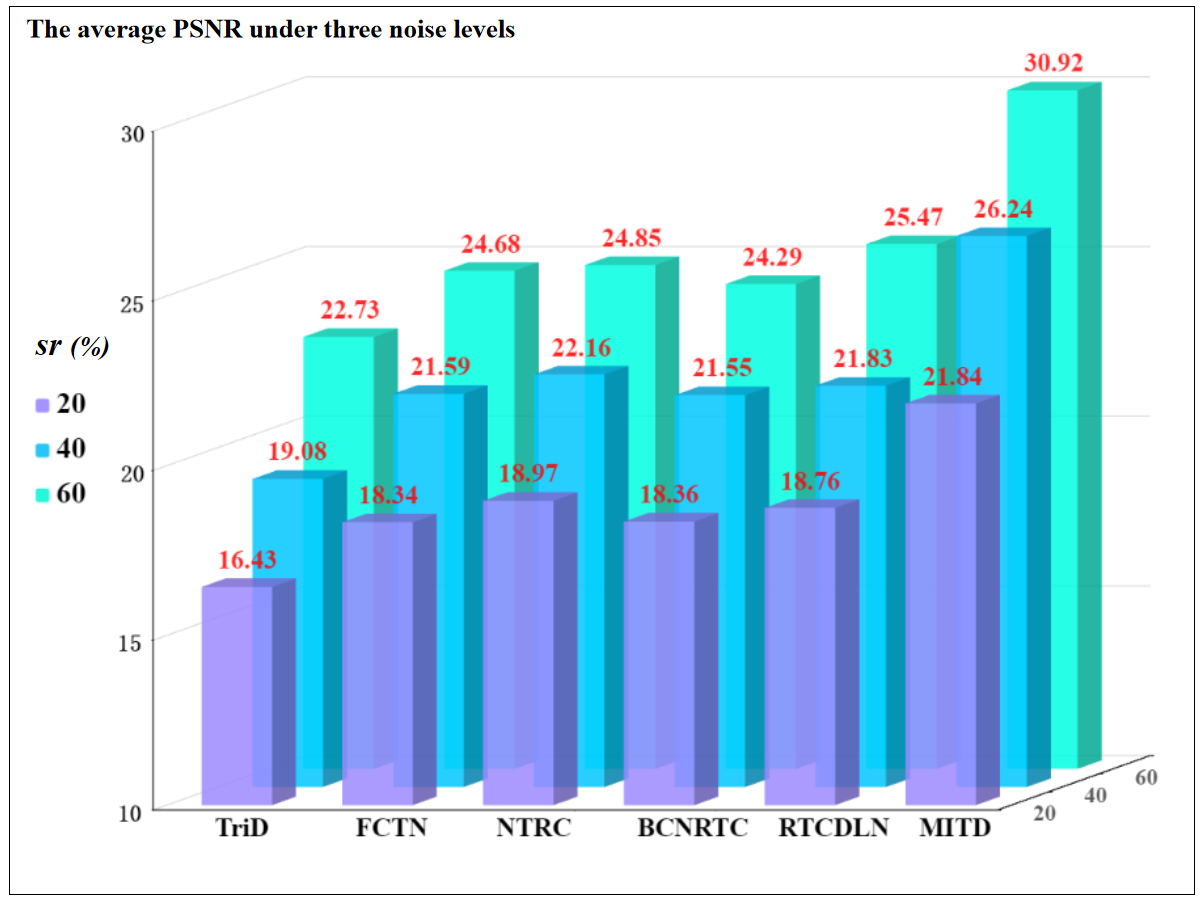}}
		
		\vspace{-5pt}
		\caption{ \small  The average PSNR  under three noise levels at different sampling rates.   }
		\label{CF5}
		\vspace{-20pt}
	\end{figure}
	From the figure,  the advantages of BCNRTC and RTCDLN become less pronounced and even inferior to those of FCTN and NTRC. This is primarily attributed to the merging of the fourth-order tensor into a third-order tensor, which disrupts the intrinsic correlations within the data, thereby diminishing the effectiveness of TNN-based RTC methods in restoring color video data. Moreover, NTRC performs better than FCTN on color video data. A possible reason is that although FCTN can establish direct connections among arbitrary factors, it inevitably introduces excessive complexity into the process, leading to degraded performance. Overall, the IMTD-based RTC method consistently achieves significant performance advantages. On one hand, it can directly handle high-order tensors with flexible rank adaptability. On the other hand, its implicit learning-based functional representation enables more effective capture and recovery of fine details in video and image data.
	
	\begin{figure*}[htbp]
		\centering
		\captionsetup[subfloat]{labelsep=none,format=plain,labelformat=empty}
		\includegraphics[width=12.5cm]{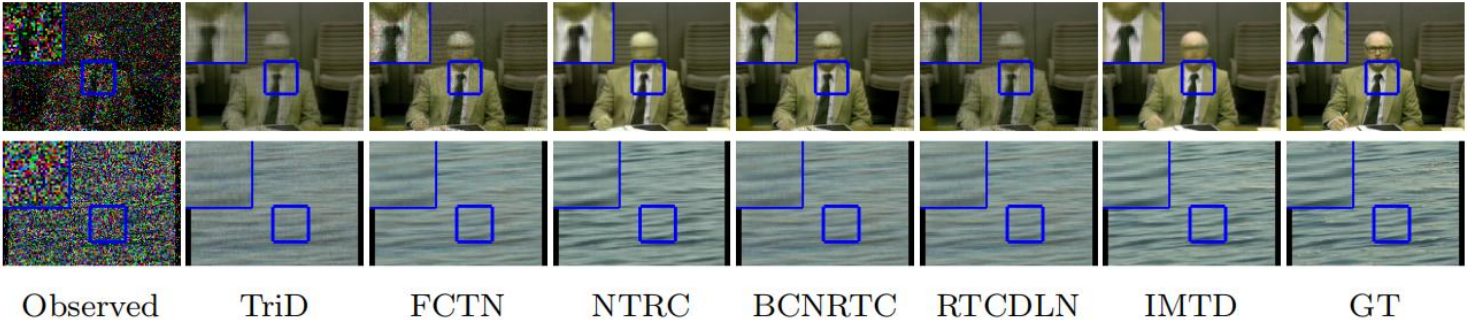}

		\vspace{-6pt}
		\caption{\small The visual comparison of restoration results on the `Man'  dataset(sr=0.4, $\sigma$=0.2) and `Ocean'  dataset(sr=0.6, $\sigma$=0.2) for each compared method. \tiny \label{fig65}}
	\end{figure*}
	
	\vspace{-13pt}
	Figure \ref{fig65} shows the performance of  compared methods on video datasets. The  reconstructed video of TriD exhibits obvious stripes and noise, primarily since it is only applicable to third-order tensors and lacks flexibility in rank selection. Both NTRC and FCTN achieved competitive performance compared to BCNRTC and RTCDLN, which indicates that the reconstruction performance of TNN-based approaches for fourth-order tensors is degraded. In contrast, our method still achieves optimal performance on video datasets. Furthermore, IMTD significantly outperforms TriD, which clearly demonstrates the advantages of our IMTD  in terms of both modeling capability and computational performance.
	
	\subsection{Comparison  on point cloud upsampling (PCU) task}
	Next, we evaluate our method on the PCU task to demonstrate its effectiveness on irregular, non-grid data. Several benchmark datasets are selected, including star, balls and heart \cite{Zhao2023}. We randomly sampled 5\% or 10\% of the points as input and applied the compared methods to upsample them.  Then, We compare IMTD with Snowflake \cite{Xiang2022}, SAPCU \cite{Zhao2023}, and LRTFR \cite{Luo2024}, using Chamfer Distance (CD) \cite{Yu2017} and F-Score \cite{Ta2015} as evaluation metrics. Table \ref{T44} records the corresponding quantitative results.
	
	\vspace{-10pt}
	\begin{table}[h]\small
		\centering
		\caption{\small The CD and F-score  for  three datasets upsampled by  compared methods.}
		\vspace{-5pt}
		\label{T44}
		\begin{tabular}{ccccccccc}
			\hline
			Dataset\qquad~& \multicolumn{2}{c}{Star} & &  \multicolumn{2}{c}{Ball} & & \multicolumn{2}{c}{Heart}  \\ \cline{2-3} \cline{5-6}  \cline{8-9} 
			Methods\qquad~ & CD & F-Score &  & CD & F-Score  &   &CD & F-Score     \\ \hline
			Snowflake \qquad~&  0.1274&  0.8496&   &0.0746 & 0.9351 &  & 0.0839 & 0.7862  \\
			SAPCU \qquad~&  0.0789&  0.9033&   &0.0527  & 0.9649 & & 0.0775 & 0.9136  \\
			LRTFR \qquad~& 0.0647 & 0.9368 & & 0.0442  & 0.9840 & & 0.0627 & 0.9615  \\
			IMTD \qquad~& \textbf{0.0581}  & \textbf{0.9642} & & \textbf{0.0435}   & \textbf{0.9866} &  & \textbf{0.0542}  & \textbf{0.9767} \\ \hline
		\end{tabular}
	\end{table}
	
	\vspace{-2pt}
	As shown in Table~\ref{T44}, the proposed IMTD consistently achieves superior performance across all three datasets. This implies that the point clouds predicted by IMTD are closer to the ground truth and exhibit better local structure and completeness. This can be primarily attributed to IMTD's architectural flexibility and its capability of exploiting multi-dimensional data features.
	The corresponding visual results are presented in Figure~\ref{fig6}.
	\begin{figure*}[htbp]
		\vspace{-0pt}
		\centering
		
		\captionsetup[subfloat]{labelsep=none,format=plain,labelformat=empty}
		
		\includegraphics[width=12.6cm]{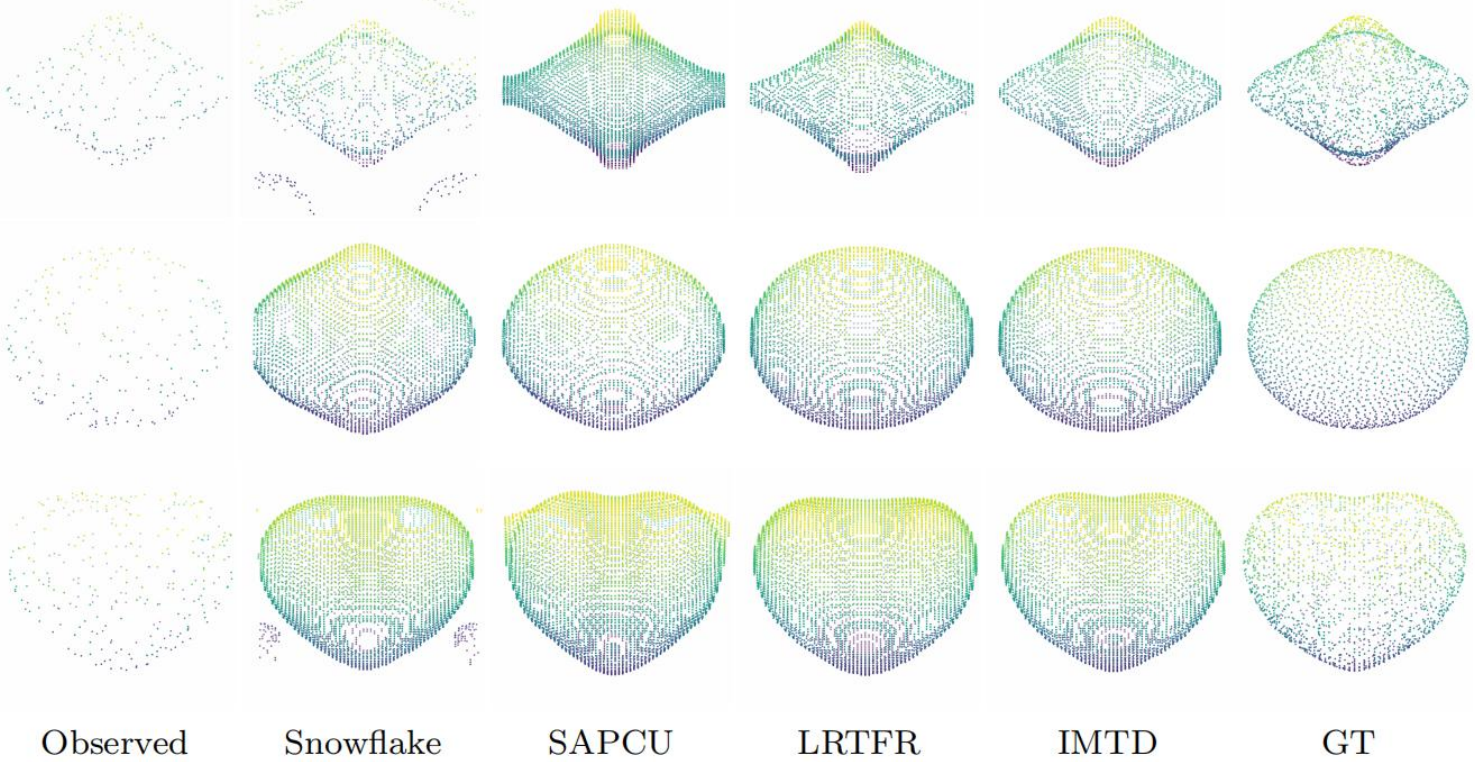}

		\vspace{-2pt}
		\caption{\small The visual comparison of compared methods on point cloud upsampling. From top to bottom: Star ($sr = 0.1$), Balls ($sr = 0.1$) and Heart ($sr = 0.05$). \tiny \label{fig6}}
		\vspace{-10pt}
	\end{figure*}
	From the figure, Snowflake's upsampling introduces additional erroneous points on both `Star' and `Heart' datasets, which may be attributed to the use of deconvolution. Moreover, its reconstruction of the `ball' dataset exhibits noticeable geometric distortions. This may be attributed to limited generalization, as Snowflake is trained on specific shape categories.  SAPCU, LRTFR, and IMTD are all based on implicit neural representations. SAPCU does not exploit the intrinsic low-rank structure of the data, resulting in noticeable deviations in its upsampling results. LRTFR adopts Tucker decomposition and parameterizes the factor matrices, achieving competitive performance. In comparison,  IMTD achieves greater preservation in overall shape and fine-grained geometric details. This advantage stems not only from IMTD’s architectural flexibility, but also from the parameterization of all  factor tensors, which enables the model to capture more refined geometric features.

	\subsection{Running time}
	Next, we compare the execution of each method, as shown in Figure \ref{CF7}. For the RTC task, IMTD exhibits increased computational efficiency than the tensor decomposition-based approaches owing to the utilization of GPU acceleration. The TNN-based methods incur higher computational cost,  due to the expensive and repeated SVD calculations during the optimization process. In the PCU task, all the compared methods leverage GPU acceleration  during training. Snowflake adopts a supervised learning framework, leading to lengthy training times but enables fast inference. SAPCU, LRTFR, and IMTD are unsupervised approaches. SAPCU directly represents point clouds using implicit neural networks, resulting in higher computational cost. LRTFR reduces dimensionality through Tucker decomposition but suffers from expensive core tensor updates. Overall, our approach demonstrates superior efficiency in both tasks.
	\begin{figure}[htbp]
		\vspace{-10pt}
		\centering
		\hspace{-0.4cm}\subfloat[RTC  (Pepper)]{\includegraphics[width=4cm]{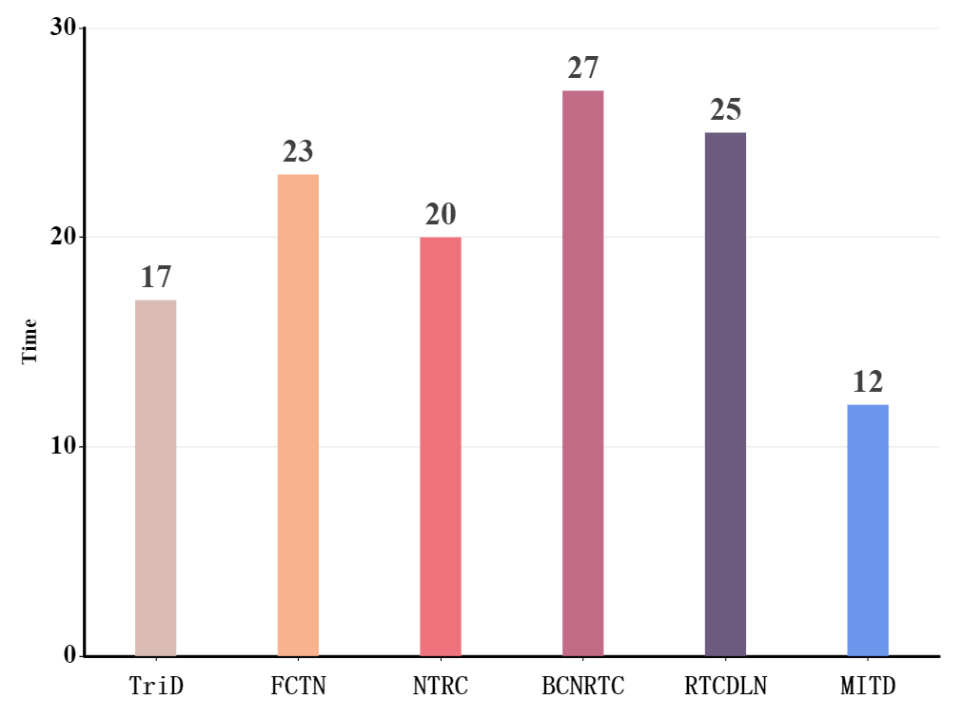}}
		\subfloat[RTC (Ocean)]{\includegraphics[width=4cm]{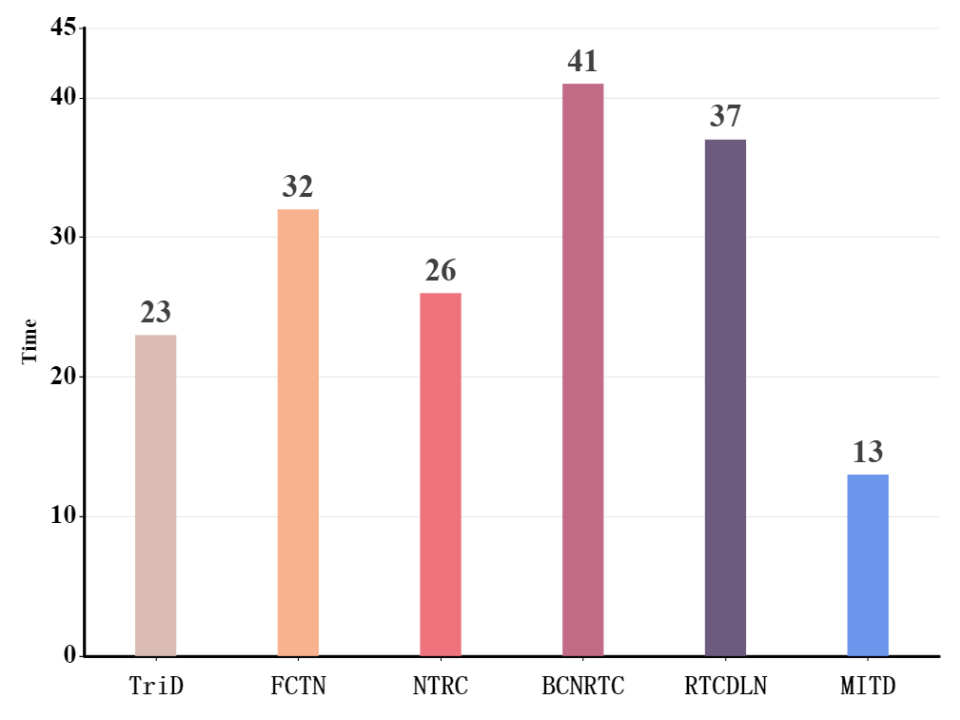}}
		\subfloat[PDU (Heart)]{\includegraphics[width=4cm]{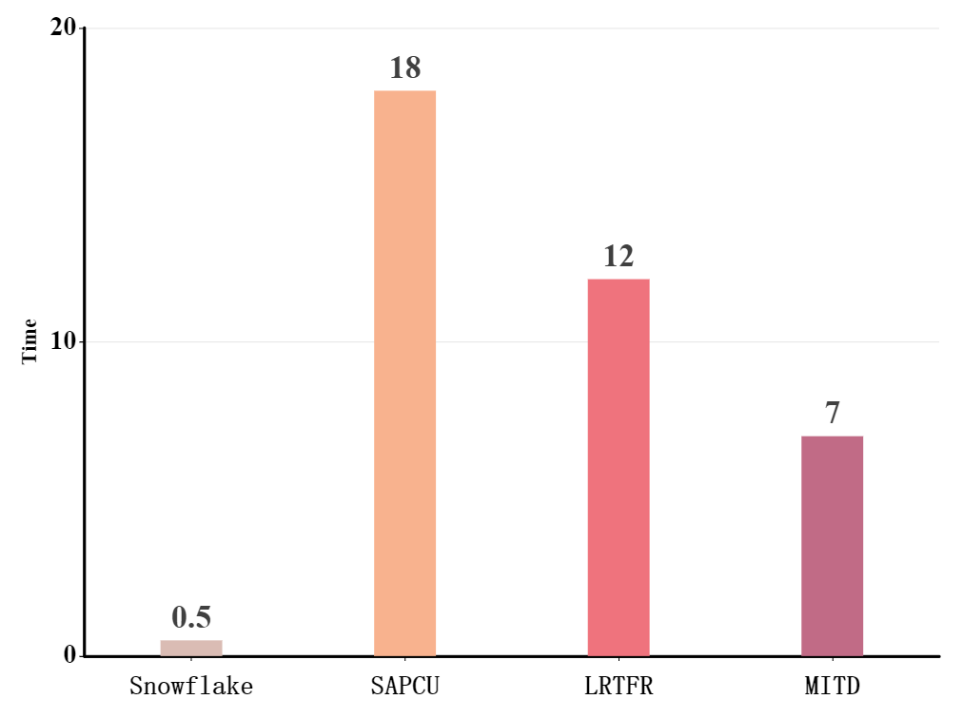}}
		
		\vspace{-3pt}
		\caption{\small  The running time of the compared methods for different tasks.   }
		\label{CF7}
		\vspace{-10pt}
	\end{figure}
	
	\section{Conclusion}\label{G}
	In this paper, we propose the Implicit Multiple Tensor Decomposition (IMTD), a novel framework that generalizes triple decomposition to tensors of arbitrary order.   In addition, it allows the short dimensions of the factor tensors to differ.  By representing the factor tensors as implicit functions, IMTD enables continuous modeling of tensor data  and  supports decomposition of irregular and non-grid data. Theoretically, we investigate the connections between IMTD and classical tensor decompositions, and develop a computational algorithm for addressing the reconstruction problem within the IMTD framework. A KL-free convergence analysis is also provided. Finally, extensive numerical experiments further validate the effectiveness and broad applicability of the proposed method.

	\section*{Acknowledgments}
	We would like to acknowledge the assistance of volunteers in putting
	together this example manuscript and supplement.
	
	\bibliographystyle{siamplain}
	\bibliography{CiteTex}
\end{document}

%% file: ex_shared.tex
\usepackage{lipsum}
\usepackage{amsfonts}
\usepackage{graphicx}
\usepackage{epstopdf}
\usepackage{algorithmic}
\ifpdf
  \DeclareGraphicsExtensions{.eps,.pdf,.png,.jpg}
\else
  \DeclareGraphicsExtensions{.eps}
\fi

\usepackage{enumitem}
\setlist[enumerate]{leftmargin=.5in}
\setlist[itemize]{leftmargin=.5in}

\newsiamremark{hypothesis}{Hypothesis}
\crefname{hypothesis}{Hypothesis}{Hypotheses}
\newsiamthm{claim}{Claim}
\newsiamremark{fact}{Fact}
\crefname{fact}{Fact}{Facts}

\headers{Implicit Multiple Tensor Decomposition }{Kunjing Yang and  Minru Bai}

\title{Implicit Multiple Tensor Decomposition \thanks{Submitted to the editors DATE.
}}

\author{Kunjing Yang\thanks{The School of Mathematics, Hunan University, Changsha, Hunan 410082, P. R. China (\email{kunjing-yang@hnu.edu.cn}).}
\and Libin Zheng\thanks{The School of Mathematics, Hunan University, Changsha, Hunan 410082, P. R. China (\email{fifholz301@hnu.edu.cn}).}
\and Minru Bai\thanks{ The Corresponding author with School of Mathematics, Hunan University, Changsha, Hunan 410082, P. R. China (\email{minru-bai@hnu.edu.cn}).}
}

\usepackage{amsopn}
